\documentclass[11pt]{article}
\usepackage{amscd, amsmath, amssymb, amsthm}
\usepackage[all,cmtip]{xy}
\usepackage[pagebackref,linktocpage]{hyperref}
\usepackage{breakurl}

\title{Bott vanishing for Fano 3-folds}
\author{Burt Totaro}
\date{  }

\def\Z{\text{\bf Z}}

\def\R{\text{\bf R}}
\def\C{\text{\bf C}}
\def\P{\text{\bf P}}

\def\N{\text{\bf N}}

\def\surj{\twoheadrightarrow}

\def\m{\mathfrak{m}}
\def\so{\mathfrak{so}}
\def\ohat{\widehat{\Omega}^1_F}

\DeclareMathOperator{\Aut}{Aut}
\DeclareMathOperator{\Gr}{Gr}
\DeclareMathOperator{\topsub}{top}
\DeclareMathOperator{\Curv}{Curv}
\DeclareMathOperator{\Pic}{Pic}
\DeclareMathOperator{\Bl}{Bl}
\DeclareMathOperator{\pt}{pt.}
\DeclareMathOperator{\Nef}{Nef}

\begin{document}
\maketitle
\newtheorem{theorem}{Theorem}[section]
\newtheorem{proposition}[theorem]{Proposition}
\newtheorem{corollary}[theorem]{Corollary}
\newtheorem{lemma}[theorem]{Lemma}
\newtheorem{conjecture}[theorem]{Conjecture}

\theoremstyle{definition}
\newtheorem{definition}[theorem]{Definition}
\newtheorem{example}[theorem]{Example}
\newtheorem{question}[theorem]{Question}

\theoremstyle{remark}
\newtheorem{remark}[theorem]{Remark}

A smooth projective variety $X$ satisfies {\it Bott vanishing }if
$$H^j(X,\Omega^i_X\otimes L)=0$$
for all $j>0$, $i\geq 0$, and all ample line bundles $L$.
This is a strong property, useful when it holds.
Combined with Riemann-Roch, Bott vanishing gives complete information
about the sections of many natural vector bundles on $X$.

Bott proved Bott vanishing for projective space. An important
generalization, by Danilov and Steenbrink, is that
every smooth projective toric variety satisfies Bott vanishing;
proofs can be found in \cite{BC, BTLM, Mustata, Fujino}.
The first non-toric Fano variety found to satisfy Bott vanishing
is the quintic del Pezzo surface \cite{Totarobott}. That paper also
analyzes Bott vanishing among
some varieties that are not
rationally connected, such as K3 surfaces.
Generalizing the quintic del Pezzo surface,
Torres showed that all GIT quotients of $(\P^1)^n$ by the action
of $PGL(2)$ satisfy Bott vanishing \cite{Torres}.
Torres's examples include one new Fano variety (not just a product) in each
even dimension.

A Fano variety $X$ that satisfies Bott vanishing is rigid, since
$H^1(X,TX)=H^1(X,\Omega^{n-1}_X\otimes K_X^*)=0$. As a result,
there are only finitely many smooth complex Fano varieties
in each dimension (up to isomorphism)
that satisfy Bott vanishing. We can view them
as a generalization of toric Fano varieties; they should
have some kind of combinatorial classification.

In this paper, we classify the smooth Fano 3-folds
that satisfy Bott vanishing. There are many more than
expected.

\begin{theorem}
\label{main}
Bott vanishing holds for exactly 37 smooth complex Fano 3-folds,
up to isomorphism.
These consist of the 18 toric Fano 3-folds and 19 others.
In Mori-Mukai's numbering, the toric Fano 3-folds are
(1.17), (2.33)--(2.36), (3.25)--(3.31), (4.9)--(4.12), (5.2)--(5.3).
The non-toric Fano 3-folds that satisfy Bott vanishing
are (2.26), (2.30), (3.15)--(3.16), (3.18)--(3.24), (4.3)--(4.8),
(5.1), and (6.1).
\end{theorem}

Here (6.1) is $\P^1$ times the quintic del Pezzo surface,
but the other 18 non-toric examples are new. Bott vanishing
fails for the quadric 3-fold,
but, surprisingly, it holds for the blow-up of the quadric
at a point, (2.30). Likewise, Bott vanishing fails for the flag manifold
$W=GL(3)/B$, but it holds for several blow-ups of $W$ such
as (3.16). In order to prove Bott vanishing in all cases of Theorem
\ref{main}, we find that the fastest
approach is to reduce systematically
to calculations on toric varieties.
Many of the calculations in characteristic zero were first made
by Belmans, Fatighenti, and Tanturri \cite{BFT, Fanography},
as explained in section \ref{fail}.

Our arguments actually prove Theorem \ref{main}
in any characteristic not 2, among the known smooth Fano 3-folds.
It is an open question whether the classification
of smooth Fano 3-folds has the same form in every characteristic;
see the discussion by Achinger-Witaszek-Zdanowicz \cite[section 7
and Appendix A]{AWZ2}. Still, among the known Fano 3-folds,
the ones that satisfy Bott vanishing
are the same in every characteristic not 2. We also find those
for which Bott vanishing persists in characteristic 2.
We give meaningful proofs for all the cohomology calculations,
although computers are useful for checking that nothing has gone wrong.

Our arguments suggest a novel conjecture on vanishing.
For every projective birational morphism $\pi\colon X\to Y$
of smooth varieties,
and every line bundle $A$ on $X$ that is ample over $Y$,
the higher direct image sheaf $R^j\pi_*(\Omega^i_X\otimes A)$ should be zero
for all $j>0$ and $i\geq 0$ (Conjecture \ref{relative}).

This work was supported by NSF grant DMS-2054553.
Thanks to Pieter Belmans, S\'andor Kov\'acs,
and Talon Stark for useful conversations. In particular,
Stark found all the cases in which Bott vanishing fails
(section \ref{fail}).

\tableofcontents

\section{Vanishing theorems}

In this section, we recall the known Bott vanishing property
for toric varieties, and we prove a variant (Proposition \ref{nef}).
The basic result (attributed to Danilov and Steenbrink)
is that for a smooth projective toric variety $X$ over a field,
we have $H^j(X,\Omega^i(L))=0$
for $j>0$, $i\geq 0$, and $L$ an ample line bundle.
Fujino proved several generalizations, such as
the following \cite[Theorem 1.3]{Fujinolog}.

\begin{theorem}
\label{log}
Let $X$ be a smooth projective toric variety over a field,
and let $D$ and E be reduced toric divisors in $X$
without common irreducible components. Then
$$H^j(X,\Omega^i_X(\log (D+E))(L-E))=0$$
for $j>0$, $i\geq 0$, and $L$ an ample line bundle.
\end{theorem}

\begin{remark}
More generally, Fujino proves Theorem \ref{log}
even when the toric variety $X$ is singular, with $\Omega^i_X$
replaced by the sheaf of reflexive differentials
$\widehat{\Omega}^i_X:=(\Omega^i_X)^{**}$. Moreover,
he allows $L$ to be an ample Weil divisor rather than a line bundle,
replacing $\Omega^i_X\otimes L$ in the statement by
the reflexive sheaf $(\Omega^i_X\otimes O(L))^{**}$.
\end{remark}

Here is a related statement
that may be new. It is helpful for some later arguments. The case $i=0$
is known \cite[p.~68, p.~74]{Fulton}.

\begin{proposition}
\label{nef}
Let $X$ be a smooth proper toric variety over a field. Let $L$
be a nef line bundle on $X$. Then $H^j(X,\Omega^i_X\otimes L)=0$
for $j>i$.
\end{proposition}

Unlike Theorem \ref{log}, Proposition \ref{nef}
does not extend to singular toric varieties using
the sheaf of reflexive differentials.
For example, Danilov found a complex projective toric variety $X$
with $H^2(X,\widehat{\Omega}^1_X)\neq 0$ \cite[Example 12.12]{Danilov}.

\begin{proof}
Use the resolution (which applies to any divisor
with simple normal crossings in a smooth variety):
$$0\to \Omega^i_X\to \Omega^i_X(\log \partial X)\to
\oplus \Omega^{i-1}_{D_j}(\log \partial D_j)\to\cdots
\to \oplus O_{D_{j_1\cdots j_i}}(\log \partial D_{j_1\cdots j_i})\to 0.$$
Here each term but the first involves logarithmic differentials
with respect to the full toric boundary, and the sums are over
all the torus-invariant subvarieties
of a given dimension in $X$. The vector bundles of logarithmic
differentials on a toric variety with respect to the full toric boundary
are all trivial \cite[p.~87]{Fulton}. Tensor this exact sequence
with a nef line bundle $L$. For each toric subvariety $D_J$ of $X$,
we have $H^j(D_J,L)=0$ for all $j>0$
\cite[p.~68, p.~74]{Fulton}.
Applying that fact to this
resolution of $\Omega^i_X\otimes L$, we conclude that
$H^j(X,\Omega^i_X\otimes L)=0$ for all $j>i$.
\end{proof}

Next, we state the Kodaira-Akizuki-Nakano (KAN) vanishing theorem
\cite[Theorem 4.2.3]{Lazarsfeld},
\cite{DI}, \cite[Theorem A.2]{AS}:

\begin{theorem}
\label{kan}
(1) Let $X$ be a smooth projective variety over a field of characteristic
zero. Then
$$H^j(X,\Omega^i\otimes L)=0$$
for all ample line bundles $L$ and all $i+j>\dim(X)$.

(2) Let $X$ be a smooth projective variety over a perfect
field $k$ of characteristic $p>0$. If $X$ lifts to $W_2(k)$
and $X$ has dimension $\leq p$, then $X$ satisfies KAN vanishing
(as in (1)). Also, if $X$ is globally $F$-split
and $X$ has dimension $\leq p+1$, then
$X$ satisfies KAN vanishing.
\end{theorem}

To check Bott vanishing in positive characteristic
for the new cases in this paper, the following lemma
is helpful.

\begin{lemma}
\label{kanlemma}
The 18 smooth Fano 3-folds studied
in this paper satisfy KAN vanishing in every characteristic.
\end{lemma}

\begin{proof}
All the known smooth Fano 3-folds in characteristic $p>0$
lift to $W_2$, which implies KAN vanishing for $p>2$. Moreover,
the 18 smooth Fano 3-folds studied in this paper are globally $F$-split
in every characteristic $p>0$, and so they satisfy KAN vanishing even
in characteristic 2. Indeed, Mehta and Ramanathan
showed that a smooth projective variety $X$ in positive characteristic
is globally $F$-split if there is a section of $-K_X$ whose zero
scheme has completion at some point isomorphic to $x_1\cdots x_n=0$,
where $n=\dim(X)$ \cite[Proposition 7]{MR}.

This condition is easy to check
for all 18 of the Fano 3-folds studied in this paper. Indeed,
in all cases except (2.26), $-K_X$ can be written as a sum
of three basepoint-free line bundles $L_1+L_2+L_3$. Then general sections
of $L_1,L_2,L_3$ are smooth. (In characteristic $p>0$, this is not
automatic from Bertini's theorem, but one checks it easily in each case.)
One also checks in each case that the resulting three
smooth divisors can be taken to meet transversely
at some point.
So $X$ is globally $F$-split.
For the rigid Fano 3-fold of type (2.26)
(section \ref{2.26}), $-K_X$ is the sum
of three effective line bundles $(H-E)+H+H$, with each one representable by
a smooth divisor in $X$. These three divisors can be taken to meet
transversely at some point. So again $X$ is globally $F$-split.
\end{proof}

By KAN vanishing, the 18 Fano 3-folds studied
in this paper have $H^j(X,TX)=H^j(X,\Omega^2_X\otimes K_X^*)=0$
for $j>1$.
These 18 Fano 3-folds are also known (in every characteristic not 2)
to be rigid, meaning that $H^1(X,TX)=0$. These results
will be part of the proof that these 18 Fano 3-folds
satisfy Bott vanishing in every characteristic not 2.

\section{Cases where Bott vanishing fails}
\label{fail}

In this section, we give Talon Stark's proof
that Bott vanishing fails for all
smooth Fano 3-folds other than the 37 in Theorem \ref{main}.
This will appear in Stark's Ph.D.\ thesis. We also give
the analogous results in positive characteristic.
Finally, we explain the related calculations
by Belmans, Fatighenti, and Tanturri.

The smooth complex Fano 3-folds were classified into 105
deformation types
by Iskovskikh and Mori-Mukai \cite{Iskovskikh1, Iskovskikh2,
MM}. A standard reference is \cite[Tables 12.3-12.6]{IP}.
The Big Table in \cite[section 6]{Araujo} and the web site
\cite{Fanography} are also convenient.

If a smooth Fano 3-fold satisfies Bott vanishing,
then the vector bundle $TX =\Omega^2_X\otimes K_X^*$ has zero cohomology
in positive degrees. In particular, the Euler characteristic
$\chi(X,TX)$ must be nonnegative. By Riemann-Roch,
this is a characteristic number of $X$, namely
$$\chi(X,TX)=\frac{1}{2}c_1^3-\frac{19}{24}c_1c_2+\frac{1}{2}c_3.$$
The invariants of Fano 3-folds tabulated in \cite{IP}
are the degree $(-K_X)^3$, the Picard number $\rho$ (equal
to the second Betti number $b_2$),
and the Hodge number $h^{2,1}$. We have
$c_1c_2=24$ (since $\chi(X,O)=c_1c_2/24$), $c_1^3=(-K_X)^3$,
and $c_3=\chi_{\topsub}(X)=2+2\rho-2h^{2,1}$. Therefore,
$\chi(X,TX)=\frac{1}{2}(-K_X)^3-18+b_2-h^{2,1}$.
This is negative for 57 of the 105 deformation classes of Fano
3-folds; so those do not satisfy Bott vanishing.
The 48 deformation classes of Fano 3-folds with $\chi(X,TX)\geq 0$ are:
(1.15)--(1.17), (2.26)--(2.36), (3.13)--(3.31), (4.3)--(4.12),
(5.1)--(5.3), (6.1), and (7.1).

Next, some Fano 3-folds are not rigid even though $\chi(X,TX)\geq 0$.
Bott vanishing implies that $H^1(X,TX)=0$, and so it
fails in such a case. Let $V_5\subset \P^6$ denote the quintic del Pezzo
3-fold (which is rigid), a smooth codimension-3 linear section
of the Grassmannian $\Gr(2,5)\subset \P^9$.
First, (2.26) is the blow-up of $V_5\subset \P^6$
along a line. The Hilbert scheme of lines on $V_5$ is isomorphic
to $\P^2$, and there are two orbits of lines under $\Aut(V_5)$.
For special lines, the normal bundle is $O(1)\oplus O(-1)$,
while for general lines the normal bundle is $O\oplus O$
\cite[Lemma 2.2.6]{KPS}. Thus, the special blow-up
is not rigid and hence does not satisfy Bott vanishing. We will
see that the general blow-up, which is rigid, does satisfy Bott vanishing
(section \ref{2.26}).

Next, (2.28) is the blow-up of $\P^3$ along a plane cubic curve; that
is clearly not rigid, because plane cubics have a 1-dimensional
moduli space. Likewise, (3.14) is the blow-up of $\P^3$ at a plane cubic
and a disjoint point,
and so it is not rigid. Next, (3.13) is the intersection
of three divisors in $(\P^2)^3$ of degrees $(1,1,0)$, $(1,0,1)$,
and $(0,1,1)$. There is a 1-dimensional moduli space of such 3-folds,
described in \cite[Lemma 5.19.7]{Araujo}, and so they are not rigid.
Finally, (7.1) is $\P^1$ times a quartic del Pezzo surface, which
is not rigid. 

Thus there are 44 rigid Fano 3-folds, up to isomorphism:
(1.15)--(1.17), (2.26)--(2.27), (2.29)--(2.36), (3.15)--(3.31), (4.3)--(4.12),
(5.1)--(5.3), and (6.1). This information can also be found
in \cite{Araujo}.

For seven of these, Bott vanishing fails, using an ample line
bundle other than $-K_X$. The failure
of Bott vanishing was known for the quadric 3-fold (1.16)
\cite[section 4.1]{BTLM},
the flag manifold $W=GL(3)/B$ (2.32) \cite[section 4.2]{BTLM},
and the quintic del Pezzo
3-fold $V_5$ (1.15)
(\cite[Lemma 7.10]{AWZ2} or \cite[section 7]{Torres}).

The earlier arguments can be simplified a bit: a Riemann-Roch
calculation suffices to disprove Bott vanishing.
(The Macaulay2 package Schubert2 is convenient for these
calculations \cite{M2}.)
Namely, for the quadric
3-fold $Q$, we have $-K_X\cong O(3)$ and $\chi(X,\Omega^2(1))=-1$.
For the flag manifold $W$, a divisor of degree $(1,1)$
in $\P^2\times \P^2$, we have $-K_W=2A+2B$ (with $A$ and $B$
the pullbacks of $O(1)$ from the two $\P^2$'s), and
$\chi(W,\Omega^2(A+B))=-1$. For the quintic del Pezzo 3-fold
$X=V_5\subset \P^6$, we have $-K_X=O(2)$ and $\chi(X,\Omega^2_X(1))
=-3$.

Bott vanishing fails for four other rigid Fano 3-folds,
as follows. (2.27) is the blow-up $X$ of $\P^3$ along a twisted cubic
curve. Here $-K_X=4H-E$ (where $H$ is the pullback of $O(1)$
from $\P^3$ and $E$ is the exceptional divisor).
But the ``smaller'' line bundle $L=3H-E$ is also ample,
and $\chi(X,\Omega^2_X\otimes L)=-2<0$.
(2.29) is the blow-up of the quadric 3-fold $Q$ along a conic. Here
$L=2H-E$ is ample and ``smaller'' than $-K_X=3H-E$, and
$\chi(X,\Omega^2_X\otimes L)=-2$. (2.31) is the blow-up of $Q$
along a line. Here $L=2H-E$ is ample and ``smaller'' than $-K_X=
3H-E$, and $\chi(X,\Omega^2_X\otimes L)=-1$. Finally,
(3.17) is a divisor
of degree $(1,1,1)$ in $\P^1\times \P^1\times \P^2$.
The line bundle $L=A+B+C$ is ample and ``smaller'' than
$-K_X=A+B+2C$, and $\chi(X,\Omega^2_X\otimes L)=-1$.

We will show that the 37 other rigid Fano 3-folds satisfy Bott vanishing.
We know this for the 18 toric Fanos (listed in Theorem \ref{main}).
We also know Bott vanishing for (6.1), the product of $\P^1$ with
the quintic del Pezzo surface \cite[Theorem 2.1, Lemma 2.3]{Totarobott}.
It remains to prove Bott vanishing for the 18 other Fano 3-folds
listed in Theorem \ref{main}.

We will prove these results for the known Fano 3-folds in every
characteristic other than 2. In characteristic 2, some of these 37 Fano
3-folds have non-reduced automorphism group scheme, essentially
because of the distinctive features of conics in characteristic 2.
(In particular, the subgroup scheme of $PGL(3)$ that preserves
a conic and a general point in $\P^2$ is not reduced, in characteristic 2.)
It follows that $X$ is not rigid in such a case, meaning that
$H^1(X,TX)\neq 0$. These cases are (2.26),
(3.15), (3.18),
(3.21), (4.3), (4.4), (4.5), and (5.1). (They are still
``set-theoretically rigid'' in the sense that they are isomorphic
to nearby varieties over an algebraically closed field.
This situation is analyzed in \cite[Theorem 0.2]{EHS}.)
Two others are rigid but do not satisfy Bott vanishing
in characteristic 2, (2.30) and (3.19).
Our arguments give that
the remaining 27 Fano 3-folds do satisfy Bott vanishing in characteristic 2.

The examples above suggest the question:

\begin{question}
\label{nefquestion}
Let $X$ be a smooth Fano variety which is rigid, meaning
that $H^1(X,TX)=0$. Is $H^j(X,\Omega^i_X\otimes K_X^*\otimes L)=0$
for all $j>0$, $i\geq 0$, and nef line bundles $L$?
\end{question}

Here rigidity is a necessary assumption, since
$H^1(X,TX)=H^1(X,\Omega^{n-1}_X\otimes K_X^*)$.
Question \ref{nefquestion} is compatible with the evidence
in this paper. Namely, in each case (above)
where Bott vanishing fails
for some rigid Fano 3-fold,
it always involves an ample
line bundle which is ``smaller'' than $-K_X$, in particular
not of the form $-K_X+L$ with $L$ nef.

In view of the isomorphism $\Lambda^iTX\cong \Omega^{n-i}_X\otimes K_X^*$,
Question \ref{nefquestion} can be rephrased in terms of the groups
$H^j(X,\Lambda^iTX\otimes L)$ for $L$ nef.
The arguments in this paper work by reducing this problem
to the case where $L$ is trivial, that is,
computing the cohomology groups of ``polyvector fields'',
$H^j(X,\Lambda^iTX)$. These groups are closely
related to the Hochschild cohomology of $X$. The cohomology
of polyvector fields was
computed for all smooth Fano 3-folds in characteristic zero
by Belmans, Fatighenti, and Tanturri \cite{BFT, Fanography}.

\section{Inductive approach to Bott vanishing}
\label{indsection}

To prove Bott vanishing for a given variety $X$ means
proving the vanishing of higher cohomology for the bundles
$\Omega^i_X$ tensored with all ample line bundles.
Intuitively, this should be hardest for the ``smallest'' ample
line bundles on $X$. In this section, we give a simple procedure
for deducing Bott vanishing for ``bigger'' ample line bundles
from smaller ones. The method works well in all our examples.

\begin{lemma}
\label{induction}
Let $X$ be a smooth projective variety over a field.
Let $D$ be a smooth divisor in $X$ that satisfies Bott vanishing.
(For example, $D$ could be a toric variety.) Let $L$ be a line
bundle on $X$ such that $L-D$ and $L$ are ample. If $X$ satisfies
Bott vanishing for $L-D$, then it satisfies Bott vanishing
for $L$.
\end{lemma}

\begin{proof}
The assumption means that $H^j(X,\Omega^i_X\otimes L(-D))=0$
for all $j>0$ and $i\geq 0$. We have an exact sequence
of coherent sheaves, $0\to O_X(-D)\to O_X\to O_D\to 0$.
Tensoring with $\Omega^i_X$ and taking cohomology,
we get an exact sequence
$$H^j(X,\Omega^i_X\otimes L(-D))\to H^j(X,\Omega^i_X\otimes L)
\to H^j(D,\Omega^i_X\otimes L).$$
Let $j>0$ and $i\geq 0$; then we are given that the first group here
is zero. In order to show that $H^j(X,\Omega^i_X\otimes L)=0$
as we want, it suffices to show that $H^j(D,\Omega^i_X\otimes L)=0$.

We have an exact sequence $0\to O_D(-D)\to \Omega^1_X|_D
\to \Omega^1_D\to 0$ of vector bundles on $D$. Taking
exterior powers, it follows that $0\to O_D(-D)\otimes
\Omega^{i-1}_D\to \Omega^i_X|_D
\to \Omega^i_D\to 0$. So the vanishing we want follows
if $H^j(D,\Omega_D^{i-1}\otimes L(-D))$
and $H^j(D,\Omega^i_D\otimes L)$ are zero. Since $L$
and $L(-D)$ are ample, both groups vanish by Bott vanishing
on $D$.
\end{proof}

\begin{remark}
Lemma \ref{induction} simplifies the proof of Bott vanishing
for the quintic del Pezzo surface \cite[Theorem 2.1]{Totarobott}.
Namely, this induction reduces the problem to the vanishing
of $H^1(X,\Omega^1_X\otimes K_X^*)=H^1(X,TX)$, which holds
because $X$ is rigid (or by a direct calculation). We will give
more details of the analogous reduction for Fano 3-folds
in the rest of the paper.
\end{remark}

\section{Higher direct images of differential forms}

Most Fano 3-folds arise as blow-ups of a simpler variety
along a point or a curve.
In order to prove Bott vanishing in such a case,
we need to analyze the cohomology
of bundles of differential forms twisted by a line bundle
on a smooth blow-up.

\begin{conjecture}
\label{relative}
Let $\pi\colon X\to Y$ be a projective birational morphism
between smooth varieties over a field, and let $A$ be a line bundle
on $X$ that is ample over $Y$. Then $R^j\pi_*(\Omega^i_X\otimes A)=0$
for all $j>0$ and $i\geq 0$.
\end{conjecture}

More generally, I expect Conjecture \ref{relative} to hold
for $Y$ with toroidal singularities. But even the question
with $Y$ smooth over a field of characteristic zero
seems not to have been raised before.
For $i+j>n=\dim(X)$ (KAN-type vanishing), the conjecture holds
in characteristic zero with no assumptions on the singularities of $Y$
\cite[Corollary 2.1.2]{Arapura}. For $i=1$ and $j=n-1$,
the Steenbrink-type vanishing theorem of
\cite[Theorem 14.1]{GKKP} is a similar statement.

The main evidence for Conjecture \ref{relative} is that it holds
for toric morphisms, by Fujino (Theorem \ref{toric}). (Toric morphisms need
not be birational, but our main interest here is in the birational
case.)
Since Conjecture \ref{relative} can be checked after completing
the base at a point, the case of toric morphisms implies
the conjecture for the blow-up of any smooth subvariety along
a smooth subvariety (Corollary \ref{higher}).
Theorem \ref{toric} implies the conjecture
for some iterated blow-ups as well.
Such iterated blow-ups occur in several examples where we check
Bott vanishing, including the hardest case,
the Fano 3-fold (5.1) (section \ref{iterated}).

\begin{theorem}
\label{toric}
(Fujino \cite[Theorem 5.2]{Fujino})
Let $\pi\colon X\to Y$ be a projective toric morphism
over a field $k$, with $X$ smooth over $k$.
Let $A$ be a line bundle
on $X$ that is ample over $Y$. Then $R^j\pi_*(\Omega^i_X\otimes A)=0$
for all $j>0$ and $i\geq 0$.
\end{theorem}

\begin{remark}
A further generalization is that Theorem \ref{toric} holds
even when the toric variety $X$ is singular, with $\Omega^i_X$
replaced by the sheaf of reflexive differentials. Moreover,
for $X$ singular, we can allow $A$ to be an ample Weil divisor,
replacing $\Omega^i_X\otimes A$ in the statement by
the reflexive sheaf $(\Omega^i_X\otimes O(A))^{**}$. For projective
toric varieties, Fujino proved Bott vanishing in this generality
\cite[Proposition 3.2]{Fujino}.
\end{remark}

\begin{corollary}
\label{higher}
Let $Y$ be a smooth variety
over a field $k$, $S$ a smooth subvariety of $Y$,
and $\pi\colon X\to Y$ the blow-up along $S$. Let $E$ be the exceptional
divisor in $X$, and let $m$ be a positive integer.
Then the higher direct image sheaves $R^j\pi_*(\Omega^i_X(-mE))$
are zero for all $j>0$, $i\geq 0$, and $m>0$.
\end{corollary}

\begin{proof}
It suffices to prove this after passing to the algebraic closure
of $k$ and completing $Y$ at a point.
So we can assume that $S$ is a linear subspace
of an affine space $Y$. In that case, the blow-up $X\to Y$
is a toric morphism. Also, the line bundle $O(-E)$ on $X$ is ample
over $Y$. So Theorem \ref{toric} implies
that $R^j\pi_*(\Omega^i_X(-mE))=0$ for all $j>0$, $i\geq 0$,
and $m>0$.
\end{proof}

For applications, we also want to compute $\pi_*(\Omega^i_X(-E))$.
We first consider the blow-up at a point. Proposition \ref{blowup}
will extend this
to the blow-up along a higher-dimensional smooth subvariety.

\begin{proposition}
\label{point}
Let $Y$ be a smooth variety of dimension $n$
over a field $k$, $p$ a $k$-point of $Y$,
and $\pi\colon X\to Y$ the blow-up at $p$. Let $E$ be the exceptional
divisor in $X$. Then
$\pi_*(\Omega^1_X(-E))$ is the subsheaf $\Omega^1_Y\otimes I_{p/Y}$
of $\Omega^1_Y$. For $i\geq 2$, $\pi_*(\Omega^{i}_X(-E))$ is equal
to $\Omega^{i}_Y$.
\end{proposition}

It would be interesting to compute explicitly the filtration of
each vector bundle $\Omega^i_Y$
by the subsheaves $\pi_*(\Omega^i_X(-mE))$ for $m\geq 0$. That is,
filter the differential forms on $Y$ by their order of vanishing
along $E\subset X$. We will need only a few cases,
Propositions \ref{point} and \ref{blowup}.

\begin{proof}
(Proposition \ref{point})
We first show that
$\pi_*(\Omega^1_X(-E))$ is the subsheaf $\Omega^1_Y\otimes I_{p/Y}
\subset \Omega^1_Y$. In other words, we want to show that a 1-form
on a neighborhood of $p$ in $Y$ vanishes at $p$ if and only if
its pullback to $X$ vanishes as a section of $\Omega^1_X$ on $E$.
Clearly, if a 1-form vanishes at $p$, then its pullback
vanishes along $E$. For the converse, let $x_1,\ldots,x_n$
be regular functions near $p$ that form a basis for $\m_p/\m_p^2$.
On one affine chart of the blow-up $X$ (which is enough to consider),
the morphism $\pi\colon X\to Y$ is given by $(x_1,u_2,\ldots,u_n)\mapsto
(x_1,x_1u_2,\ldots,x_1u_n)$. So $dx_1$ pulls back to $dx_1$
and $dx_i$ for $2\leq i\leq n$ pulls back to $x_1du_i+u_idx_1$.
Restricting to sections of $\Omega^1_X$
on the exceptional divisor $E=\{x_1=0\}$, these sections
become $dx_1,u_2dx_1,\ldots,u_ndx_1$. Since these are linearly independent
over $k$, the 1-forms on $Y$ whose pullback vanishes along $E$
are only those that vanish at $p$.

Next, for $i\geq 2$,
let us show that the subsheaf $\pi_*(\Omega^{i}_X(-E))$
of $\Omega^{i}_Y$ is equal to $\Omega^{i}_Y$. That is, we want
to show that the pullback of every $i$-form on $Y$ vanishes as a section
of $\Omega^{n-1}_X$ on $E$. This follows from the previous paragraph's
calculation: the 1-forms $dx_1,\ldots,dx_n$ pull back
to $dx_1,u_2dx_1,\ldots,u_ndx_1$ as sections of $\Omega^1_X$ on $E$,
and any wedge product of $i\geq 2$ such 1-forms vanishes on $E$
(since $dx_1\wedge dx_1=0$).
\end{proof}

\section{The Fano 3-folds (2.30) and (3.19)}
\label{quadric}

We now check Bott vanishing for the first new cases, (2.30) and (3.19).
These are the blow-up of the quadric 3-fold at one point,
or at two non-collinear points. 
Part of the argument involves reducing to properties of toric varieties,
although other parts are special to the quadric. Most of the later cases
will reduce completely to results on toric varieties.

The proof of Bott vanishing for (2.30) and (3.19)
works in characteristic not 2. Bott vanishing actually fails
for these two varieties in characteristic 2.

Consider case (2.30) first.
Here $X$ is the blow-up of the smooth quadric 3-fold $Q$ at a point $p$.
(It is also the blow-up of $\P^3$ along a conic.)
The Picard group of $X$ is $\Z\{H,E\}$, where $H$ denotes
the pullback of $O(1)$ from $Q$ and $E$ is the exceptional divisor.
We have $-K_X=3H-2E$.

For each smooth complex Fano 3-fold, Coates, Corti, Galkin,
and Kasprzyk determined the nef cone \cite{CCGK}.
For the Fano 3-folds
in this paper, we will compute
the nef cone again. One reason is to make sure that
the proofs work in any characteristic. Another reason is that
it is convenient for our arguments
to find explicit generators
for the cone of curves.

Consider case (2.30) first.
Here $X$ is the blow-up of the smooth quadric 3-fold $Q$ at a point $p$.
(It is also the blow-up of $\P^3$ along a conic.)
The Picard group of $X$ is $\Z\{H,E\}$, where $H$ denotes
the pullback of $O(1)$ from $Q$ and $E$ is the exceptional divisor.
We have $-K_X=3H-2E$.
Let $C$ be a line
in $E\cong \P^2$.
Let $D$ be the strict transform of a line on $Q$ through $p$.
We have the intersection numbers:

\begin{tabular}{c|rr}
& $C$ & $D$  \\
\hline
$H$& 0 & 1 \\
$E$& $-1$& 1\\
\end{tabular}

So the dual basis to $C,D$ is given by $H-E,H$.
Here $H-E$ and $H$ are basepoint-free, hence nef, giving
contractions of $X$ to $\P^3$ and $Q$.
(Sections of the line bundle $H-E$ on $X$ correspond to sections
of $H$ on $Q$ that vanish at the point $p$. Basepoint-freeness
of $H-E$ on $X$ follows from the fact
that $p$ is cut out by sections of $H$, as a subscheme of $Q$.)
It follows that the closed cone of curves $\overline{\Curv(X)}$ is
$\R^{\geq 0}\{ C,D\}$, and the nef cone is spanned by $H-E$
and $H$. More strongly, 
the monoid of nef classes in $\Pic(X)$
is generated by $H$ and $H-E$. 

Since $(2H-E)\cdot C=(2H-E)\cdot D=1$,
the line bundle $2H-E$ is ample, and every ample line bundle
is $2H-E$ plus a nef divisor, hence $2H-E$ plus an $\N$-linear
combination of $H$ and $H-E$. A general divisor in the linear
system $|H|$ on $X$ is a quadric surface $\P^1\times \P^1$, and a general
divisor in $|H-E|$ is also a quadric surface. Since these are both toric,
Lemma \ref{induction} reduces Bott vanishing to the single
ample line bundle $2H-E$.

The vanishing of $H^j(X,\Omega^i_X\otimes L)$ is clear whenever
$X$ is a Fano 3-fold, $j>0$, $L$ is ample, and $i$ is 0 or 3.
Indeed, the case $i=3$ is Kodaira vanishing (part of KAN
vanishing, Lemma \ref{kanlemma}). The case $i=0$
follows from Kodaira vanishing since $-K_X$ is ample. So,
throughout this paper, we only need to consider Bott vanishing
with $i$ equal to 1 or 2.

For $i=2$, Proposition \ref{point} gives that
\begin{align*}
H^j(X,\Omega^2_X\otimes O(2H-E))
&= H^j(Q,\pi_*(\Omega^2_X(-E))\otimes O(2))\\
&= H^j(Q,\Omega^2_Q(2)).
\end{align*}
For $k$ of characteristic 2, Macaulay2 shows that $h^1(Q,\Omega^2_Q(2))=1$,
and so $h^1(X,\Omega^2_X\otimes O(2H-E))=1$. It follows that
Bott vanishing fails for (2.30) in characteristic 2.

Let us analyze the tangent bundle of the quadric $Q\subset \P^4$.
Here $Q$ is defined by a quadratic form $q$ on $k^5$,
which we can assume is $x_0x_2+x_1x_3+x_4^2$.
Since we assume that the characteristic of the base field $k$
is not 2, the associated bilinear form
on $k^5$ is nondegenerate.
Think of $Q$
as the space of isotropic lines $L$ in $k^5$. By the Euler sequence,
the tangent bundle of $\P^4$ at a point $[L]$ has $T\P^4(-1)
\cong k^5/L$. It follows that $TQ(-1)$ at a point $[L]$ in $Q$
is $L^{\perp}/L$. As a result, there is a canonical nondegenerate
symmetric bilinear form on $TQ(-1)$, $TQ(-1)\otimes TQ(-1)\to O_Q$.
So $\Omega^1_Q\cong TQ(-2)$. We also have
$TQ\cong \Omega^2_Q\otimes K_Q^*=\Omega^2_Q(3)$.

The vector bundle $\Omega^2_Q(2)=TQ(-1)$ has zero cohomology
in all degrees. To check this by hand, first note that $L:=O_Q(-1)$
has zero cohomology in all degrees, by reducing to the known cohomology
of line bundles on $\P^4$. Then observe that $L^{\perp}\subset O_Q^{\oplus
5}$ has zero cohomology in all degrees, using that $O_Q^{\oplus 5}/L^{\perp}$
is isomorphic to $L^*=O_Q(1)$. So $\Omega^2_Q(2)=TQ(-1)=L^{\perp}/L$
has zero cohomology
in all degrees, as we want. We remark that $h^1(Q,\Omega^2_Q(1))=1$,
which checks again that $Q$ itself does not satisfy Bott vanishing.

For $i=1$, Proposition \ref{point} gives that
$$H^j(X,\Omega^1_X\otimes O(2H-E))
= H^j(Q,\pi_*(\Omega^1_X(-E))\otimes O(2)).$$
The Proposition also gives a short exact sequence
$$0\to \pi_*(\Omega^1_X(-E))\to \Omega^1_Q\to \Omega^1_Q|_p\to 0.$$
So Bott vanishing for $i=1$ and the ample line bundle $2H-E$ on $X$
follows if $H^j(Q,\Omega^1_Q(2))=0$ for $j>0$ and
$H^0(Q,\Omega^1_Q(2))$ maps onto its fiber at $p$. By the isomorphism
$\Omega^1_Q\cong TQ(-2)$ above, we can rephrase the problem
in terms of the tangent bundle, which seems easier to visualize.
Namely, we want to show that $H^j(Q,TQ)=0$ for $j>0$
and that $TQ$ is spanned at $p$ by its global sections.

The fact that $H^j(Q,TQ)=0$ for $j>0$ is immediate from $Q$
being a rigid Fano variety. Namely, rigidity means
that $H^1(Q,TQ)=0$, and KAN vanishing implies
that $H^j(Y,TQ)=0$ for $j\geq 2$.
(KAN vanishing holds for the quadric 3-fold $Q$
in any characteristic, by the same argument as in Lemma \ref{kanlemma}.)
The fact that $TQ$ is spanned at $p$
by global sections follows in characteristic zero
from the homogeneity of the quadric $Q$. In any characteristic,
one can check by hand that the subspace of the Lie algebra
$\so(5)$ that fixes the point $p$ in $Q$ has codimension 3,
which proves that the map $\so(5)\to H^0(Q,TQ)$ maps onto
the fiber of $Q$ at $p$. Explicitly, for the quadratic
form $q=x_0x_2+x_1x_3+x_4^2$ over any field,
$\so(5)$ is the space of $5\times 5$ matrices
$$\begin{pmatrix}X_1 & X_2 &b_1\\X_3 & X_4 & b_2\\
a_1 & a_2 & 0
\end{pmatrix}$$
with each $X_i$ a $2\times 2$ matrix such that $X_4=-X_1^t$,
$X_2$ and $X_3$ are skew-symmetric with zeros on the diagonal,
$b_1=-2a_2^t$, and $b_2=-2a_1^t$ \cite[sections 23.4, 23.6]{Borel}.
We read off that the stabilizer of $p=[1,0,0,0,0]$ in $Q$
has codimension 3, in any characteristic.
That completes the proof (in characteristic not 2)
of Bott vanishing for (2.30), the blow-up
of $Q$ at a point.

We can also handle (3.19), the blow-up $X$
of the quadric 3-fold $Q$ at two non-collinear points $p_1$
and $p_2$. Again, Bott vanishing fails for $X$ in characteristic 2,
and so we assume that the base field $k$ has characteristic not 2.
Let $E_1$ and $E_2$ be the exceptional divisors
over $p_1$ and $p_2$. We have $\Pic(X)=\Z\{H,E_1,E_2\}$.
Since $-K_Q=3H$, we have $-K_W=3H-2E_1-2E_2$.

Let $C_1\subset E_1\cong \P^2$ and $C_2\subset E_2\cong\P^2$
be lines. Let $G_1$ be the strict transform
of a line in $Q$ through $p_1$, and $G_2$ the strict transform
of a line in $Q$ through $p_2$. We have the intersection numbers:

\begin{tabular}{c|rrrr}
& $C_1$ & $C_2$ & $G_1$ & $G_2$ \\
\hline
$H$  & 0 & 0 & 1 & 1\\
$E_1$& $-1$& 0 & 1 & 0\\
$E_2$& 0 & $-1$& 0 & 1\\
\end{tabular}

It follows that the dual cone
to $\R^{\geq 0}\{ C_1,C_2,G_1,G_2\}$ is spanned by
$H$, $H-E_1$, $H-E_2$, and $H-E_1-E_2$. (This calculation
is immediate for Magma \cite{Magma}.)
These four divisors are basepoint-free, hence nef,
giving contractions of $X$ to $Q$, $\P^3$, $\P^3$, and $\P^2$.
It follows that $\overline{\Curv(X)}=\R^{\geq 0}\{ C_1,C_2,
G_1,G_2\}$, and the nef cone is spanned by $H$, $H-E_1$,
$H-E_2$, and $H-E_1-E_2$. This is the first non-simplicial cone
we have encountered. Using Magma, we check
that the nef monoid in $\Pic(X)$ is also
generated by those four divisors.

Let $M=2H-E_1-E_2$; then $M$ has degree 1 on $C_1$, $C_2$,
$G_1$, and $G_2$. So $M$ is ample, and every ample line bundle
is $M$ plus a nef divisor, hence $M$ plus an $\N$-linear
combination of $H$, $H-E_1$, $H-E_2$, and $H-E_1-E_2$.

A general divisor in $|H|$ is a quadric surface $\P^1\times \P^1$,
which is toric. A general divisor in $|H-E_1|$ or $|H-E_2|$
is a quadric surface blown up at a point, which is also toric.
A general divisor in $|H-E_1-E_2|$ is a quadric surface
blown up at 2 non-collinear points, which is also toric.
By Lemma \ref{induction}, this reduces Bott vanishing on $X$
to the ample line bundle $M$.

That is, we want to show that $H^j(X,\Omega^1_X\otimes M)$
and $H^j(X,\Omega^2_X\otimes M)$ are zero for $j>0$.
By Proposition \ref{point}, for $\Omega^2_X$,
it is equivalent to show that
$H^j(Q,\Omega^2_Q(2))=0$ for $j>0$. We already showed this
in characteristic not 2,
in analyzing (2.30) above. As mentioned there, $h^1(Q,\Omega^2_Q(2))=1$
in characteristic 2, and so Bott vanishing actually fails
for (3.19) in characteristic 2.

For $\Omega^1_X$, Proposition \ref{point}
gives that $R\pi_*(\Omega^1_X\otimes M)\cong \Omega^1_Q(2)\otimes I_{p_1
\cup p_2/Q}$. So we have an exact sequence
$$H^0(Q,\Omega^1_Q(2))\to \Omega^1_Q|_{p_1}\oplus \Omega^1_Q|_{p_2}
\to H^1(X,\Omega^1_X\otimes L)\to H^1(Q,\Omega^1_Q(2)).$$
So it suffices to show that $\Omega^1_Q(2)$ has zero cohomology
in positive degrees, and that its global sections map onto
the sum of its fibers at $p_1$ and $p_2$. As shown above,
$\Omega^1_Q(2)$ is isomorphic to the tangent bundle of $Q$.
Here $H^j(Q,TQ)=0$ for $j>0$ because $Q$ is a rigid Fano,
as discussed above. It remains to show that the global sections
of the tangent bundle of $Q$ map onto the direct sum
of its fibers at $p_1$ and $p_2$.

In characteristic zero,
this holds because the group $O(5)$ over $\C$,
acting on the quadric $Q\subset \P^4$,
acts transitively on pairs of non-collinear points. Indeed, in terms
of the given quadratic form on $\C^5$,
such a pair corresponds to a hyperbolic plane in $\C^5$,
and the Witt Extension Theorem gives that $O(5)$ acts transitively
on the set of hyperbolic planes in $\C^5$ \cite[Theorem 8.3]{EKM}.
In any characteristic, we can assume that $q=x_0x_2+x_1x_3+x_4^2$,
$p_1=[1,0,0,0,0]$, and $p_2=[0,0,1,0,0]$. By the description
of the Lie algebra $\so(5)$ above, the subspace
that fixes $p_1$ and $p_2$ in $Q$ has codimension $3+3=6$,
in any characteristic. That is,
the map $\so(5)\to H^0(Q,TQ)$ maps onto the direct
sum of its fibers at $p_1$ and $p_2$.
That completes the proof of Bott vanishing (in characteristic not 2)
for (3.19).

\section{Higher direct images of differential forms, continued}

We now compute the higher direct images of differential
forms twisted by a line bundle for a blow-up along
a smooth subvariety (not just a point). We will use this
to check Bott vanishing for the remaining Fano 3-folds,
since each of them is the blow-up of a simpler variety
along a curve.

\begin{proposition}
\label{blowup}
Let $Y$ be a smooth variety
over a field $k$, $F$ a smooth subvariety of $Y$,
and $\pi\colon X\to Y$ the blow-up along $F$. Let $E$ be the exceptional
divisor in $X$, and let $m$ be a positive integer.
Then the higher direct image sheaves $R^j\pi_*(\Omega^i_X(-mE))$
are zero for all $j>0$ and $i\geq 0$. For $j=0$ and $m=1$,
$\pi_*(\Omega^i_X(-E))$ is the subsheaf of $\Omega^i_S$ whose sections
restricted to $\Omega^i_S|_F$ lie in the image of the product map
$\Omega^{i-2}_F|_Y\otimes_{O_Y}\Lambda^2 N_{F/Y}^*\to \Omega^i_F|_Y$.
\end{proposition}

To describe the subsheaf $\pi_*(\Omega^i_X(-E))$ in more detail:
the vector bundle $\Omega^i_Y|_F$ is filtered with quotients
$\Omega^i_F$ (on top), then $\Omega^{i-1}_F\otimes N_{F/Y}^*$,
then $\Omega^{i-2}_F\otimes \Lambda^2 N_{F/Y}^*$, and so on.
The quotient sheaf $\Omega^i_Y/\pi_*(\Omega^i_X(-E))$ is a vector bundle
on $F$, consisting of the top two steps of the filtration. That is,
it is an extension of $\Omega^i_F$ by $\Omega^{i-1}_F\otimes N_{F/Y}^*$.

\begin{proof}
The line bundle $O(-E)$ on $X$ is ample over $Y$. So Corollary
\ref{higher} gives that $R^j\pi_*(\Omega^i_X(-mE))$ is zero
for all $j>0$, $i\geq 0$, and $m>0$.

Next, we want to describe the subsheaf $\pi_*(\Omega^i_X(-E))$
of $\Omega^i_Y$. That is, which $i$-forms on $Y$ pull back to $i$-forms
on $X$ that vanish (as sections of $\Omega^i_X$) on $E$?
We want to show that the answer is the subsheaf of $\Omega^i_Y$
of $i$-forms whose restriction to $\Omega^i_F|_Y$ lies in the image
of $\Omega^{i-2}_Y|_F\otimes \Lambda^2 N_{F/Y}^*\to \Omega^i_Y|_F$.
By replacing $k$ by its algebraic closure and completing
in the neighborhood of a point,
it suffices to check this statement when $F=A^b\times\{0\}$ and $X=A^b\times
A^{n-b}$, as above. Let $x_1,\ldots,x_{n-b}$ be coordinates
on $A^{n-b}$ and $y_1,\ldots,y_b$ coordinates on $F$.

Clearly $i$-forms on $Y$ that vanish in $\Omega^i_Y|_F$
pull back to $i$-forms that vanish in $\Omega^i_X|_E$.
So it suffices to consider $i$-forms on $Y$ that are $O_F$-linear
combinations of wedge products of $dx_i$'s and $dy_j$'s.
To see whether such a form pulls back to one that vanishes on $E$,
we can work in a single affine chart of the blow-up of $Y$ along $F$,
as in the proof of Proposition \ref{point}. The blow-up map
$\pi\colon X\to Y$ is given in this chart by $(x_1,u_2,
\ldots,u_{n-b},y_1,\ldots,y_b)\to (x_1,x_1u_2,\ldots,x_1u_{n-b},
y_1,\ldots,y_b)$. So, pulling back and restricting to $\Omega^i_X|_E$,
$dx_1$ maps to $dx_1$, $dx_i$ maps to $u_idx_1$ for $1\leq i\leq n-b$,
and $dy_i$ pulls back to $dy_i$. It follows that the $i$-forms
that pull back to zero in $\Omega^i_X|_E$ are spanned by those with
at least two $dx_i$ factors. That is the statement we want.
\end{proof}

Let us spell out what Proposition \ref{blowup} says
in the main case used in this paper,
the blow-up of a 3-fold along a curve.

\begin{corollary}
\label{curve}
Let $\pi\colon X\to Y$
be the blow-up of a smooth 3-fold along a smooth curve $F$
over a field. Then $R^j\pi_*(\Omega^i_X(-mE))=0$
for $j>0$, $i\geq 0$, and $m>0$. Also, $\pi_*(\Omega^1_X(-E))$
is the kernel of the surjection $\Omega^1_Y\to \Omega^1_Y|_F$,
and $\pi_*(\Omega^2_X(-E))$ is the kernel of the surjection $\Omega^2_Y
\to \Omega^1_F\otimes N_{F/Y}^*$.
\end{corollary}

The following lemma works very efficiently to prove the base case
of Bott vanishing in most of our examples.

\begin{lemma}
\label{ci}
Let $X$ be the blow-up of a smooth projective toric variety $Y$
along a smooth codimension-2 subvariety $F$ that is a complete intersection
$S_1\cap S_2$ in $Y$.  Let $E$ be the exceptional
divisor on $X$.
\begin{enumerate}
\item
Suppose that $-K_Y$, $-K_Y-S_1$,
and $-K_Y-S_2$ are ample, and $-K_Y-S_1-S_2$ is nef.
Then $H^j(X,\Omega^1_X(-K_X))=0$ for $j>0$.
\item Let $L$ be a line bundle on $Y$
such that $L$, $L-S_1$, and $L-S_2$ are ample,
and $L-S_1-S_2$ is nef.
Then $H^j(X,\Omega^1_X(\pi^*(L)-E))=0$ for $j>0$.
\item
Let $L$ be a line bundle on $Y$
such that $L$ and $L-S_1$ are ample.
Suppose that $S_1$ is a toric variety
(not necessarily torically embedded in $Y$), $(L-S_2)|_{S_1}$ is ample,
and $(L-S_1-S_2)|_{S_1}$ is nef.
Then $H^j(X,\Omega^1_X(\pi^*(L)-E))=0$ for $j>0$.
\end{enumerate}
\end{lemma}

\begin{proof}
We have $-K_X=\pi^*(-K_Y)-E$, and so part 1 follows from part 2.
Let us prove part 2.
By Corollary \ref{curve}, it suffices to show
that (1) $H^j(Y,\Omega^1_Y(L))=0$ for $j>0$;
(2) $H^0(Y,\Omega^1_Y(L))
\to H^0(F,\Omega^1_Y(L))$ is surjective;
and (3) $H^j(F,\Omega^1_Y(L))=0$ for $j>0$.

First, we know (1) by Bott vanishing on the toric variety $Y$.
Next, consider the exact sequence $H^j(Y,\Omega^1_Y(L))
\to H^j(S_1,\Omega^1_Y(L))\to
H^{j+1}(Y,\Omega^1_Y(L-S_1))$. By Bott vanishing on $Y$,
we know that $H^j(Y,\Omega^1_Y(L))=0$
and $H^j(Y,\Omega^1_Y(L-S_1))=0$
for $j>0$. So $H^j(S_1,\Omega^1_Y(L))=0$ for $j>0$
and $H^0(Y,\Omega^1_Y(L))\to H^0(S_1,\Omega^1_Y(L))$
is surjective. Next, consider the exact sequence
$H^j(S_1,\Omega^1_Y(L))\to H^j(F,\Omega^1_Y(L))
\to H^{j+1}(S_1,\Omega^1_Y(L-S_2))$.
Statements (2) and (3) follow if we can show that $H^j(S_1,
\Omega^1_Y(L-S_2))=0$ for $j>0$.

To do that, consider the exact sequence
$H^j(Y,\Omega^1_Y(L-S_2))
\to H^j(S_1,\Omega^1_Y(L-S_2))\to
H^{j+1}(Y,\Omega^1_Y(L-S_1-S_2))$. By Bott vanishing on $Y$,
we know that $H^j(Y,\Omega^1_Y(L-S_2))=0$ for $j>0$.
Since $L-S_1-S_2$ is nef on the toric variety $Y$, Proposition
\ref{nef} gives that
$H^j(Y,\Omega^1_Y(L-S_1-S_2))=0$
for $j>1$. It follows that $H^j(S_1,\Omega^1_Y(L-S_2))=0$
for $j>0$.

To prove part 3, we reduce as before to showing that
$H^j(S_1, \Omega^1_Y(L-S_2))=0$ for $j>0$. Consider the exact
sequence $0\to O(-S_1)_{S_1}\to \Omega^1_Y|_{S_1}
\to \Omega^1_{S_1}\to 0$. Tensoring with $L-S_2$ and taking
cohomology, we have an exact sequence
$H^j(S_1,L-S_1-S_2)\to H^j(S_1,\Omega^1_Y(L-S_2))
\to H^j(S_1,\Omega^1_{S_1}(L-S_2))$.
Since $S_1$ is a toric variety, the first group is zero for $j>0$
since $(L-S_1-S_2)|_{S_1}$ is nef (Proposition \ref{nef}).
The last group
is zero for $j>0$ since $(L-S_2)|_{S_1}$ is ample. It follows
that $H^j(S_1,\Omega^1_Y(L-S_2))=0$ for $j>0$, which completes
the proof.
\end{proof}

\section{First blow-up along a curve: (2.26)}
\label{2.26}

Most Fano 3-folds are blow-ups of simpler varieties along a smooth
curve. We now prove Bott vanishing in one such case. Although part
of the method will apply to later examples, this case is relatively
far from toric varieties and hence requires individual handling.

In case (2.26), $X$ is the blow-up of the quintic del Pezzo threefold
$Y:=V_5\subset \P^6$
along a general line. We will instead use another description,
mentioned by Mori-Mukai \cite[Table 3]{MM}: $X$ is the blow-up
of the quadric 3-fold $Q$ along a twisted cubic curve $F$.
(It is easier to work with $Q$ than with $V_5$.) We assume that
$F$ is general, in the sense that the $\P^3$ it spans
is transverse to $Q\subset \P^4$; otherwise, Bott vanishing fails,
as discussed in section \ref{fail}. We work in characteristic not 2,
since $X$ is not rigid (hence does not satisfy Bott vanishing)
in characteristic 2.

We have $\Pic(X)=\{H,E\}$, and $-K_X=3H-E$. 
Let $C\cong \P^1$
be a fiber of the exceptional divisor $E\to F$.
Let $D$ be the strict transform
of a line in $Q$ that meets $F$ in two points.
We have the intersection numbers:

\begin{tabular}{c|rr}
& $C$ & $D$  \\
\hline
$H$& 0 & 1 \\
$E$& $-1$& 2\\
\end{tabular}

So the dual basis to $C,D$
is given by $2H-E,H$. 
The line bundles $2H-E$ and $H$ are basepoint-free, hence nef,
corresponding to contractions of $X$ to $V_5\subset \P^6$
and to $Q$. It follows that $\overline{\Curv(X)}=\R^{\geq 0}\{ C,D\}$,
and the nef cone is spanned by $H$ and $2H-E$. More strongly,
the nef monoid in $\Pic(X)$
is generated by $H$ and $2H-E$. Since $-K_X\cdot C=1$ and $-K_X\cdot D=1$,
every ample line bundle on $X$ is $-K_X$ plus a nef divisor,
hence $-K_X$ plus an $\N$-linear combination of $H$ and $2H-E$.

A general divisor in $|H|$ on $X$ is the blow-up of a quadric
surface $\P^1\times \P^1$ at three points with no three collinear,
hence to the quintic del Pezzo surface. That is not toric, but
it satisfies Bott vanishing. A general divisor in $|2H-E|$ is a quartic
del Pezzo surface, which does not satisfy Bott vanishing. However,
$|H-E|$ consists of one smooth quadric surface, which is toric and hence
satisfies Bott vanishing. Since $2H-E=H+(H-E)$, Lemma \ref{induction}
along with the previous paragraph
reduces Bott vanishing to the single line bundle $-K_X$.

It is easy to show that $H^j(X,\Omega^2_X\otimes K_X^*)=H^j(X,TX)$
vanishes for $j>0$. First, this is zero for $j\geq 2$ by
Kodaira-Akizuki-Nakano (KAN) vanishing (Lemma \ref{kanlemma}).
That shows in general that deformations of smooth Fano varieties
are unobstructed. The fact that $H^1(X,TX)=0$ in this case amounts
to the known rigidity of this Fano 3-fold.

It remains to show that $H^j(X,\Omega^1_X\otimes K_X^*)=0$ for $j>0$.
Corollary \ref{curve} gives an exact sequence
$$H^{j-1}(F,\Omega^1_Q(-K_Q))\to H^j(X,\Omega^1_X(-K_X))
\to H^j(Q,\Omega^1_Q(-K_Q))\to H^j(F,\Omega^1_Q(-K_Q)).$$
So the desired vanishing would follow if (1) $H^j(Q,\Omega^1_Q(-K_Q))=0$
for $j>0$, (2) $H^0(Q,\Omega^1_Q(-K_Q))\to H^0(F,\Omega^1_Q(-K_Q))$
is surjective, and (3) $H^1(F,\Omega^1_Q(-K_Q))=0$. By section
\ref{quadric}, $\Omega^1_Q$ is isomorphic to $TQ(-2)$
and $-K_Q$ is $O(3)$; so
we can restate (1)--(3) in terms of the vector bundle $TQ(1)$.

Let $S$ be the intersection of $Q\subset \P^4$ with the $\P^3$ spanned
by $F$; so $S$ is a smooth quadric surface, isomorphic to $\P^1\times \P^1$.
Write $A$ and $B$ for the two pullbacks of $O(1)$ to $S$.
Then $F$ is linearly equivalent to $A+2B$ or $2A+B$ on $S$; without loss
of generality, we can assume that $F\sim 2A+B$ on $S$.
Although $F$ is not a complete intersection in $Q$,
we can work with the chain $F\subset S\subset Q$, as follows.

First, we have an exact sequence $0\to TS\to TQ|_S\to O_S(A+B)\to 0$
on $S$, and $TS\cong O(2A)\oplus O(2B)$ since $S=\P^1\times \P^1$.
Since $O(1)$ restricted to $S$ is $O(A+B)$, we read off that
$H^j(S,TQ(1))=0$ for $j>0$. Since $Q$ is a rigid Fano variety,
we have $H^j(Q,TQ)=0$ for $j>0$. Then the exact sequence
$$H^j(Q,TQ)\to H^j(Q,TQ(1))\to H^j(S,TQ(1))\to H^{j+1}(Q,TQ)$$
 implies that $H^j(Q,TQ(1))=0$
for $j>0$, which is statement (1). Since $H^1(Q,TQ)=0$, the 
sequence also gives that $H^0(Q,TQ(1))\to H^0(S,TQ(1))$ 
is surjective.

Next, since $O_S(1)\sim A+B$ and $F\sim 2A+B$ on $S$, we have
an exact sequence:
\[
\label{2.26.1}
H^j(S,TQ(1))\to H^j(F,TQ(1))\to H^{j+1}(S,TQ|_S(-A)).\tag{*}
\]
By the previous paragraph's description of $TQ|_S$, we have
an exact sequence
$0\to O(A)\oplus O(2B-A)\to TQ|_S(-A)\to O(B)\to 0$
on $S$. By Kodaira vanishing (or just the K\"unneth formula)
on $S=\P^1\times \P^1$, we read off that $H^j(S,TQ|_S(-A))=0$
for $j>0$. So the exact sequence \eqref{2.26.1}
gives that $H^1(F,TQ(1))=0$
(statement (3)) and also that $H^0(S,TQ(1))\to H^0(F,TQ(1))$
is surjective.
Together with the previous paragraph, this proves (2).
Thus we have shown that $H^j(X,\Omega^1_X\otimes K_X^*)=0$
for $j>0$. This completes the proof of Bott vanishing for (2.26).

\section{The Fano 3-fold (3.24)}

We now prove Bott vanishing for the Fano 3-fold (3.24).
Here $X$ is the blow-up of the flag manifold $W$, a smooth divisor
of degree $(1,1)$ in $\P^2\times\P^2$, along a fiber
of the first projection.
We will instead use
a different description, mentioned by Mori-Mukai \cite[Table 3]{MM}.
Namely, $X$ is the blow-up of $Y=\P^1\times \P^2$ along a curve $F$
of degree $(1,1)$.
The advantage of this description
is that it expresses $X$ as the blow-up of a toric
variety. The proof works in any characteristic.

Writing $f_1$ and $f_2$ for the projections to $\P^1$
and $\P^2$, let $A=f_1^*O(1)$ and $B=f_2^*O(1)$.
Let $E$ be the exceptional divisor in $X$.
Then $\Pic(X)=\Z\{A,B,E\}$ and $-K_X=2A+3B-E$.
Let $C$ be a fiber of $E\to F$.
Let $D_1$ be the strict transform of a curve $\P^1\times \pt$
that meets $F$. Let $D_2$ be the strict transform
of a curve $\pt\times\text{line}$ that meets $F$.
We have the intersection numbers:

\begin{tabular}{c|rrr}
& $C$ & $D_1$ & $D_2$ \\
\hline
$A$& 0 & 1 & 0\\
$B$& 0 & 0 & 1 \\
$E$& $-1$& 1 & 1\\
\end{tabular}

It follows that the dual basis to $C,D_1,D_2$ is given
by $A+B-E,A,B$. These three line bundles
are basepoint-free, hence nef, giving contractions
of $X$ to $\P^2$, $\P^1$, and $\P^2$. It follows
that $\overline{\Curv(X)}=\R^{\geq 0}\{C,D_1,D_2\}$,
and the nef cone is spanned by $A+B-E,A,B$. More strongly,
the nef monoid in $\Pic(X)$
is generated by $A+B-E$, $A$, and $B$.

We have $-K_X=2A+3B-E$. Let $M=2A+2B-E$; then $M\cdot C=1$,
$M\cdot D_1=1$, and $M\cdot D_2=1$. So $M$ is ample, and 
every ample line bundle on $X$ is $M$ plus a nef divisor,
hence $M$ plus an $\N$-linear combination of $A+B-E$,
$A$, and $B$.

A general divisor in $|A|$ or $|A+B-E|$ 
is the blow-up
of $\P^2$ at one point, which is toric. A general divisor
in $|B|$ is the blow-up of $\P^1\times \P^1$ at one point, which is also
toric. By Lemma \ref{induction} plus the previous paragraph,
this reduces Bott vanishing on $X$ to the single line bundle $M$.

For $M$ and $\Omega^1_X$,
we need to show that $H^j(X,\Omega^1_X\otimes M)=0$
for $j>0$, where $M=2A+2B-E$.
Let $L$ be the line bundle $2A+2B$ on $Y=\P^1\times\P^2$.
The nef cone of $Y$ is spanned by $A$ and $B$.
The curve $F$ is a complete intersection $S_1\cap S_2$ in $Y$,
with $S_1\sim B$ and $S_2\sim A+B$. 
By Lemma \ref{ci}, it suffices to show that $L$,
$L-S_1$, and $L-S_2$ are ample, and that $L-S_1-S_2$ is nef.
By the nef cone of $Y$, these things are true, with $L-S_1-S_2\sim A$.

It remains to prove Bott vanishing for the ample line bundle $M$
and $\Omega^2_X$. Here $L=-K_Y-B$.
By Corollary \ref{curve}, it suffices to show:
(1) $H^j(Y,\Omega^2_Y\otimes L)=0$ for $j>0$,
(2) $H^0(Y,TY(-B))\to H^0(F,N_{F/Y}(-B))$ is surjective,
and (3) $H^1(F,N_{F/Y}(-B))=0$.
Here (1) is immediate from Bott vanishing on the toric variety $Y
=\P^1\times \P^2$. Next, the tangent bundle of $Y$
is $f_1^*(T\P^1)\oplus f_2^*(T\P^2)$, and so
$TY(-B)=f_1^*(T\P^1)\otimes O(-B)\oplus f_2^*(T\P^2(-1))$.

By the chain of inclusions
$F\subset \P^1\times \P^1\subset \P^1\times \P^2$,
the normal bundle $N_{F/Y}$ is an extension
$0\to O(2)\to N_{F/Y}\to O(1)\to 0$, and so
$N_{F/Y}\cong O(2)\oplus O(1)$. It follows that
$N_{F/Y}(-B)\cong O(1)\oplus O$.
This proves (3). For (2), we will just use the second
summand of $TY$; that is, we will prove that
$H^0(\P^2,T\P^2(-1))=H^0(Y,f_2^*(T\P^2(-1)))$
maps onto $H^0(F,N_{F/Y}(-B))$.

Here $F$ embeds as a line in $\P^2$, so we can first use
that $H^0(\P^2,T\P^2(-1))$ maps onto $H^0(F,T\P^2(-1))$,
since $H^1(\P^2,T\P^2(-2))=0$.
Next, the composition $T\P^2|_F
\subset TY|_F\surj N_{F/Y}$ is an isomorphism,
because the derivative of $f_1\colon F\to \P^1$ is nonzero
at every point. So $H^0(F,T\P^2(-1))$ maps isomorphically
to $H^0(F,N_{F/Y}(-B))$.
Thus (2) is proved.
This completes the proof of Bott vanishing for
the Fano 3-fold (3.24).

\section{The Fano 3-folds (3.15), (3.16), (3.18),
(3.20), (3.21), (3.22), (3.23)}

In these seven cases with Picard number 3,
we prove Bott vanishing efficiently
by relating each Fano variety to a toric variety.
In each case except (3.20), the Fano variety
is the blow-up of a smooth toric variety along a smooth curve.

We next prove Bott vanishing for (3.15), the blow-up
of the quadric 3-fold $Q$ along a disjoint line
and conic. We will instead use
a different description, mentioned by Mori-Mukai \cite[Table 3]{MM}.
Namely, $X$ is the blow-up of $Y=\P^1\times \P^2$ along a smooth curve $F$
of degree $(2,2)$.
(The advantage of this description
is that it expresses $X$ as the blow-up of a toric
variety along a curve.) Here $X$ is not
rigid in characteristic 2, and so we work over a field
of characteristic not 2. We can take
$F$ to be given by the equations $y_0y_1=y_2^2, x_0y_1=x_1y_0=0$
in $\P^1\times \P^2=\{([x_0,x_1],[y_0,y_1,y_2])\}$.

Here $\Pic(X)=\Z\{A,B,E\}$
and $-K_X=2A+3B-E$. Let $C$ be a fiber of $E\to F$.
Let $D_1$ be the strict transform of a curve $\P^1\times\pt$
that meets $F$ in one point. Let $D_2$ be the strict transform
of a curve $\pt\times\text{line}$ that meets $F$ in two points.
We have the intersection numbers:

\begin{tabular}{c|rrr}
& $C$ & $D_1$ & $D_2$ \\
\hline
$A$& 0 & 1 & 0\\
$B$& 0 & 0 & 1\\
$E$&$-1$& 1 & 2\\
\end{tabular}

It follows that the dual basis to $C,D_1,D_2$ is $A+2B-E,A,B$.
The line bundles $A+2B-E$, $A$, and $B$ are basepoint-free,
hence nef. Namely, they contract $X$ to the quadric 3-fold $Q$,
$\P^1$, and $\P^2$.
It follows
that $\overline{\Curv(X)}=\R^{\geq 0}\{C,D_1,D_2\}$,
and the nef cone is spanned by $A+2B-E$, $A$,
and $B$. More strongly, the nef monoid in $\Pic(X)$ is generated
by these three divisors. Since $-K_X\cdot C=-K_X\cdot D_1
=-K_X\cdot D_2=1$, every ample line bundle is $-K_X$
plus a nef divisor, hence $-K_X$ plus an $\N$-linear combination
of $A+2B-E$, $A$, and $B$.

A general divisor in $|A|$ is $\P^2$ blown up at two points,
which is toric. A general divisor in $|B|$ is a quadric
surface $\P^1\times\P^1$ blown up at two non-collinear points,
which is toric. Finally, a general
divisor in $|A+2B-E|$ is $\P^2$ blown up at 4 points with no two
collinear. So it is the quintic del Pezzo
surface (which is not toric). All these surfaces satisfy Bott vanishing.
By Lemma \ref{induction}, this reduces Bott vanishing on $X$
to the ample line bundle $-K_X$.

Since $X$ is rigid (in characteristic not 2),
we have $H^j(X,\Omega^2_X\otimes K_X^*)
=H^j(X,TX)=0$ for $j>0$. It remains to show that
$H^j(X,\Omega^1_X\otimes K_X^*)=0$ for $j>0$. Lemma \ref{ci}
does exactly what we need. Namely, the curve $F$ in $Y=\P^1\times \P^2$
was defined (above) as a complete intersection $S_1\cap S_2$
with $S_1\sim A+B$ and $S_2\sim 2B$ on $Y$. Because $-K_Y=2A+3B$,
the line bundles $-K_Y$, $-K_Y-S_1=A+2B$, and $-K_Y-S_2=2A+B$ are ample,
and $-K_Y-S_1-S_2=A$ is nef.
By Lemma \ref{ci}, it follows
that $H^j(X,\Omega^1_X\otimes K_X^*)=0$
for $j>0$.
That proves Bott vanishing for (3.15).

We now prove Bott vanishing for the Fano 3-fold (3.16).
The proof works in any characteristic.
Namely, $X$ is the blow-up of the toric variety $Y=\Bl_{\text{pt.}}\P^3$
along the strict transform $F$
of a twisted cubic through the point $p$ in $\P^3$.
We have $\Pic(Y)=\Z\{H,E_1\}$, where $E_1$ is the exceptional divisor
over $p$, and $\Pic(X)=\Z\{H,E_1,E_2\}$, where $E_2$
is the exceptional divisor over $F$. (We write $E_1$ in $X$
for the strict transform of $E_1$ in $Y$.) We have
$-K_Y=4H-2E_1$ and $-K_X=4H-2E_1-E_2$. Let $C_1$ be the strict
transform of a line in $E_1\cong \P^2\subset Y$ through the point
$E_1\cap F$. Let $C_2$ be a fiber of $E_2\to F$.
Let $D$ be the strict transform of a line in $\P^3$ through $p$
that meets the twisted cubic in another point.
We have the intersection numbers:

\begin{tabular}{c|rrr}
& $C_1$ & $C_2$ & $D$ \\
\hline
$H$& 0 & 0 & 1\\
$E_1$&$-1$& 0 & 1 \\
$E_2$& 1 &$-1$& 1\\
\end{tabular}

The dual basis to $C_1,C_2,D$ is given by
$H-E_1,2H-E_1-E_2,H$.
These three divisors are basepoint-free,
giving contractions of $X$ to $\P^2$, $\P^2$, and $\P^3$. 
It follows that $\overline{\Curv(X)}
=\R^{\geq 0}\{C_1,C_2,D\}$, and the nef cone is spanned
by $H-E_1,2H-E_1-E_2,H$.
More strongly, the nef monoid in $\Pic(X)$ is generated
by these three divisors. Since $-K_X\cdot C_1=1$, $-K_X\cdot C_2=1$,
and $-K_X\cdot D=1$, every ample line bundle on $X$
is $-K_X$ plus a nef divisor, hence $-K_X$ plus an
$\N$-linear combination of $H-E_1,2H-E_1-E_2,H$.

A general divisor in $|H|$ or $|H-E_1|$ is isomorphic to
$\P^2$ blown up at 3 non-collinear points, which is toric.
A general divisor in $|2H-E_1-E_2|$ is a quadric surface
$\P^1\times \P^1$ blown up at one point, which is toric.
By Lemma \ref{induction}, this reduces Bott vanishing on $X$
to the ample line bundle $-K_X$.

For $-K_X$ and $\Omega^2_X$, this is easy: since $X$ is rigid,
we have $H^1(X,\Omega^2_X\otimes K_X^*)=H^1(X,TX)=0$, and KAN vanishing
gives that $H^j=0$ for $j\geq 2$ (Lemma \ref{kanlemma}). It remains
to prove Bott vanishing for $-K_X$ and $\Omega^1_X$,
meaning that $H^j(X,\Omega^1_X\otimes K_X^*)=0$ for $j>0$.
By Corollary \ref{curve}, this would follow if
(1) $H^j(Y,\Omega^1_Y(-K_Y))=0$
for $j>0$, (2) $H^0(Y,\Omega^1_Y(-K_Y))\to H^0(F,\Omega^1_Y(-K_Y))$
is surjective, and (1) $H^1(F,\Omega^1_Y(-K_Y))=0$.

We cannot apply Lemma \ref{ci}, because the curve $F$ is not a complete
intersection in $Y$ (because the twisted cubic is not a complete
intersection in $\P^3$). Instead, note that the twisted cubic
is a curve of bidegree $(1,2)$ on a smooth quadric surface $\P^1\times \P^1$
in $\P^3$. Let $S$ be the strict transform of that surface
in $Y=\Bl_p\P^3$; then we can work with the chain $F\subset S\subset Y$.
The surface $S$ is isomorphic to $\Bl_p(\P^1\times \P^1)$.

Statement (1) is immediate from Bott vanishing on the toric variety $Y$.
Next, consider the exact sequence
$$H^j(Y,\Omega^1_Y(-K_Y))
\to H^0(S,\Omega^1_Y(-K_Y))\to H^{j+1}(Y,\Omega^1_Y(-K_Y-S)).$$
Here $\Pic(Y)=\Z\{H,E\}$, $-K_Y=4H-2E$, and $S\sim 2H-E$ on $Y$.
The nef cone of $Y$ is spanned by $H$ and $H-E$. So $-K_Y-S=2H-E$
is ample, and hence $H^j(Y,\Omega^1_Y(-K_Y-S))=0$ for $j>0$
by Bott vanishing on $Y$ again.
It follows that $H^j(S,\Omega^1_Y(-K_Y))=0$ for $j>0$, and that
$H^0(Y,\Omega^1_Y(-K_Y))\to H^0(S,\Omega^1_Y(-K_Y))$ is surjective.

Next, consider the exact sequence:
\[
\label{3.16.1}
H^j(S,\Omega^1_Y(-K_Y))
\to H^j(F,\Omega^1_Y(-K_Y))\to H^{j+1}(S,\Omega^1_Y|_S(-K_Y-F)). \tag{*}
\]
Here $\Pic(S)\cong \Z\{A,B,E\}$, $F\sim 2A+B-E$ on $S$,
and $-K_Y|_S=4A+4B-2E$. Also, $\Nef(S)=\R^{\geq 0}\{A,B,A+B-E\}$.
So $-K_Y-F=2A+3B-E$ is ample on $S$.

To analyze $\Omega^1_Y$ restricted to $S$, use the exact sequence
\[
\label{3.16.2}
0\to O_S(-S)\to \Omega^1_Y|_S\to \Omega^1_S\to 0, \tag{**}
\]
where
$O_S(-S)=-2A-2B+E$. The surface
$S=\Bl_p(\P^1\times \P^1)$
is a toric variety, although not torically embedded in $Y$. So
Bott vanishing on $S$ gives that $H^j(S,\Omega^1_S(-K_Y-F))=0$.
Also, $-S|_S-K_Y-F=B$, which is nef. So
$H^j(S,O(-S-K_Y-F))=0$ for $j>0$ by Proposition \ref{nef}.
By \eqref{3.16.2},
we conclude that $H^j(S,\Omega^1_Y|_S(-K_Y-F))=0$ for $j>0$.
By \eqref{3.16.1},
we deduce that $H^1(F,\Omega^1_Y(-K_Y))=0$ (which is statement
(3)) and that $H^0(S,\Omega^1_Y(-K_Y))\to H^0(F,\Omega^1_Y(-K_Y))$
is surjective. That completes the proof of statement (2).
Thus we have shown that $H^j(X,\Omega^1_X(-K_X))=0$ for $j>0$.
This completes the proof of Bott vanishing for the Fano 3-fold (3.16).

We now prove Bott vanishing for (3.18). We work in characteristic not 2,
since $X$ is not rigid in characteristic 2. Here $X$ is the blow-up
of $\P^3$ along a disjoint line $F_1$ and conic $F_2$. We have
$\Pic(X)=\Z\{H,E_1,E_2\}$ and $-K_X=4H-E_1-E_2$.
Let $C_1$ be a fiber of $E_1\to F_1$ and $C_2$ a fiber
of $E_2\to F_2$. Let $D$ be the strict transform
of a line meeting the line in one point and the conic
in two points (which exists). We have the following intersection numbers.

\begin{tabular}{c|rrr}
& $C_1$ & $C_2$ & $D$\\
\hline
$H$  & 0 & 0 & 1\\
$E_1$& $-1$& 0 & 1\\
$E_2$& 0 & $-1$& 2\\
\end{tabular}

It follows that the dual basis to $C_1,C_2,D$ is
$H-E_1$, $2H-E_2$, $H$. These three divisors
are basepoint-free, hence nef, giving contractions
of $X$ to $\P^1$, the quadric 3-fold $Q$, and $\P^3$.
So $\overline{\Curv(X)}=\R^{\geq 0}\{ C_1,C_2,D\}$,
and the nef cone is spanned by $H-E_1,2H-E_2,H$.
More strongly, the nef monoid in $\Pic(X)$ is generated
by these three divisors. Also, $-K_X$ has degree 1
on $C_1$, $C_2$, and $D$. So every ample line bundle
is $-K_X$ plus a nef divisor, hence $-K_X$ plus
an $\N$-linear combination of $H-E_1$, $2H-E_2$, $H$.

A general divisor in $|H-E|$ is the blow-up of $\P^2$
at 2 points, which is toric. A general divisor in $|2H-E_2|$
is the blow-up of a quadric surface $\P^1\times \P^1$
at 2 points, which is toric.
A general divisor in $|H|$ is the blow-up of $\P^2$
at 3 non-collinear points, which is toric.
By Lemma \ref{induction}, Bott vanishing for $X$
reduces to the single line bundle $-K_X$.

Since $X$ is rigid, we know that $H^j(X,\Omega^2_X\otimes
K_X^*)=H^j(X,TX)=0$ for $j>0$. It remains to show
that $H^j(X,\Omega^1_X\otimes
K_X^*)=0$ for $j>0$. Let $Y$ be the blow-up of $\P^3$ along a line,
so that $X$ is the blow-up of $Y$ along the disjoint conic $F_2$.
This description
has the advantage that $Y$ is toric. Here $-K_Y=4H-E_1$,
and $F_2$ is a complete intersection $S_1\cap S_2$ in $Y$
with $S_1\sim H$ and $S_2\sim 2H$. So $-K_Y$, $-K_Y-S_1=3H-E_1$,
and $-K_Y-S_2=2H-E_1$ are ample, and $-K_Y-S_1-S_2=H-E_1$ is nef.
By Lemmas \ref{ci} and \ref{nef}, it follows that
$H^j(X,\Omega^1_X\otimes
K_X^*)=0$ for $j>0$.
That completes the proof of Bott vanishing for (3.18).

We now prove Bott vanishing for (3.20). The proof
works in any characteristic. Here $X$ is the blow-up
of the quadric 3-fold $Q$ along two disjoint lines,
$F_1$ and $F_2$. We have $\Pic(X)=\Z\{H,E_1,E_2\}$
and $-K_X=3H-E_1-E_2$. Let $C_i$ be a fiber of $E_i\to F_i$,
for $i=1,2$. Let $D$ be the strict transform of a line
in $Q$ meeting $F_1$ and $F_2$, which exists.
We have the following intersection numbers.

\begin{tabular}{c|rrr}
& $C_1$ & $C_2$ & $D$\\
\hline
$H$  & 0 & 0 & 1\\
$E_1$& $-1$& 0 & 1\\
$E_2$& 0 & $-1$& 1\\
\end{tabular}

It follows that the dual basis to $C_1,C_2,D$ is
$H-E_1$, $H-E_2$, $H$. These three divisors
are basepoint-free, hence nef, giving contractions
of $X$ to $\P^2$, $\P^2$, and $Q$.
So $\overline{\Curv(X)}=\R^{\geq 0}\{ C_1,C_2,D\}$,
and the nef cone is spanned by $H-E_1,H-E_2,H$.
More strongly, the nef monoid in $\Pic(X)$ is generated
by these three divisors. Also, $-K_X$ has degree 1
on $C_1$, $C_2$, and $D$. So every ample line bundle
is $-K_X$ plus a nef divisor, hence $-K_X$ plus
an $\N$-linear combination of $H-E_1$, $H-E_2$, $H$.

A general divisor in $|H-E_1|$ or $|H-E_2|$ is the blow-up of
a quadric surface $\P^1\times \P^1$ at one point,
which is toric. A general divisor in $|H|$
is the blow-up of a quadric surface at two points,
which is also toric.
By Lemma \ref{induction}, Bott vanishing for $X$
reduces to the single line bundle $-K_X$.

Since $X$ is rigid, we know that $H^j(X,\Omega^2_X\otimes
K_X^*)=H^j(X,TX)=0$ for $j>0$. It remains to show
that $H^j(X,\Omega^1_X\otimes
K_X^*)=0$ for $j>0$. We cannot apply our usual method,
since $X$ is not the blow-up of a toric variety along a curve.
Instead, we will view $X$ as a hypersurface in a toric variety $Z$:
the blow-up of $\P^4$ along two disjoint lines.
We have $\Pic(Z)=\Z\{H,E_1,E_2\}$, with $-K_Z=5H-2E_1-2E_2$,
and $X$ is linearly equivalent to $2H-E_1-E_1$ on $Z$.

Let $L=3H-E_1-E_2$ on $Z$ (so that
$L$ restricted to $X$ is $-K_X$).
Consider the exact sequence $0\to O(L-X)|_X\to \Omega^1_Z(L)|_X
\to \Omega^1_X(L)\to 0$. 
The line bundle
$O(L-X)|_X=O(H)|_X=O(K_X+(4H-E_1-E_2))$
has zero cohomology
in degrees $>0$ on $X$, by Kodaira vanishing. (In fact, this holds
in any characteristic, by reducing to Kodaira vanishing on the toric
variety $Z$.) So the vanishing of cohomology of degree $>0$
for $\Omega^1_X(L)$ would follow from the same statement
for $\Omega^1_Z(L)|_X$. Now consider the exact sequence
$0\to \Omega^1_Z(L-X)\to \Omega^1_Z(L)
\to \Omega^1_Z(L)|_X\to 0$. Here $\Omega^1_Z(L)$ has zero cohomology
in degrees $>0$ by Bott vanishing on the toric variety $Z$
(since $L$ is ample), and $\Omega^1_Z(L-X)$ has zero cohomology
in degrees $>1$ by Proposition \ref{nef} (since $L-X=H$
is nef). It follows that $\Omega^1_Z(L)|_X$ has zero
cohomology in degrees $>0$.
That completes the proof of Bott vanishing for (3.20).

We now prove Bott vanishing for (3.21). We work in characteristic not 2,
since $X$ is not rigid in characteristic 2. Here $X$ is the blow-up
of $\P^1\times\P^2$ along a curve $F$ of degree $(2,1)$.
We can take
$F=\{y_2=0,x_1y_0^2=x_0y_1^2\}$ in $\P^1\times \P^2=\{ ([x_0,x_1],
[y_0,y_1,y_2])\}$.
Thus $F$ is contained in $\P^1\times l$, for a line $l$ in $\P^2$,
We have $\Pic(X)=\Z\{A,B,E\}$ and $-K_X=2A+3B-E$.
Let $C$ be a fiber of $E\to F$. Let $q$ be a point on the line $l$,
and let $D_1$ be the strict transform of $\P^1\times q$.
Let $D_2$ be the strict transform of $p\times l$, for some point
$p$ in $\P^1$.
We have the intersection numbers:

\begin{tabular}{c|rrr}
& $C$ & $D_1$ & $D_2$\\
\hline
$A$  & 0 & 1 & 0 \\
$B$  & 0 & 0 & 1\\
$E$  & $-1$& 1 & 2\\
\end{tabular}

We compute that the dual basis to $C,D_1,D_2$
is $A+2B-E, A,B$. 
These three divisors are basepoint-free, hence nef,
using the equation for $F$ above. They give
contractions of $X$ to
a nodal 3-fold in $\P^6$, $\P^1$, and $\P^2$.
It follows that $\overline{\Curv(X)}=\R^{\geq 0}\{C,D_1,D_2\}$,
and the nef cone is spanned by the three divisors mentioned.
By the intersection numbers, 
the nef monoid is also generated by those three divisors.
The line bundle $-K_X=2A+3B-E$ has degree 1 on all three curves
$C,D_1,D_2$. So every ample line bundle is $-K_X$
plus a nef divisor, hence $-K_X$ plus an $\N$-linear combination
of the three divisors mentioned.

A general divisor in $|A|$ is $\P^2$ blown up at 2
points, which is toric. A general divisor in $|B|$
is the blow-up of $\P^1\times \P^1$
at one point, which is toric. A general divisor
in $|A+2B-E|$ is $\P^2$ blown up at 4 points with no 3 collinear,
hence a quintic del Pezzo surface. That is not toric, but it satisfies
Bott vanishing.
By Lemma \ref{induction}, that reduces Bott vanishing for (3.21)
to the single line bundle $-K_X$.

Since $X$ is rigid, we know that $H^j(X,\Omega^2_X\otimes K_X^*)
=H^j(X,TX)$ is zero for $j>0$. It remains to show that
$H^j(X,\Omega^1_X\otimes K_X^*)$ is zero for $j>0$.
The curve $F$ is a complete intersection
$S_1\cap S_2$ in $Y=\P^1\times \P^2$, with $S_1\sim B$
and $S_2\sim A+2B$. Since $-K_Y=2A+3B$, we see that $-K_Y$,
$-K_Y-S_1$, and $-K_Y-S_2$ are ample. By Lemma \ref{ci},
the desired vanishing would follow if
$H^j(Y,\Omega^1_Y(-K_Y-S_1-S_2))=0$ for $j>1$.
Here $-K_Y-S_1-S_2=A$ is nef, and so this follows
from Proposition \ref{nef}.
That completes
the proof of Bott vanishing for (3.21).

We now prove Bott vanishing for (3.22). The proof
works in any characteristic. Here $X$ is the blow-up
of $\P^1\times\P^2$ along a conic $F$ in $p\times \P^2$,
for a point $p$ in $\P^1$. 
We can take
$F=\{x_1=0,y_0y_1=y_2^2\}$ in $\P^1\times \P^2=\{ ([x_0,x_1],
[y_0,y_1,y_2])\}$.
We have $\Pic(X)=\Z\{A,B,E\}$ and $-K_X=2A+3B-E$.
Let $C$ be a fiber of $E\to F$. Let $D_1$ be the strict
transform of a line in $p\times \P^2$. Let $D_2$ be the strict
transform of $\P^1\times q$, for a point $q$ in the conic.
We have the intersection numbers:

\begin{tabular}{c|rrr}
& $C$ & $D_1$ & $D_2$\\
\hline
$A$  & 0 & 0 & 1 \\
$B$  & 0 & 1 & 0\\
$E$  & $-1$& 2 & 1\\
\end{tabular}

We compute that the dual basis to $C,D_1,D_2$
is $A+2B-E, B,A$. 
These three divisors are basepoint-free, hence nef,
using the equation for $F$ above. They give
contractions of $X$ to
the cone in $\P^6$ over the Veronese surface in $\P^5$, to $\P^2$,
and to $\P^1$.
It follows that $\overline{\Curv(X)}=\R^{\geq 0}\{C,D_1,D_2\}$,
and the nef cone is spanned by the three divisors mentioned.
By the intersection numbers, 
the nef monoid is also generated by those three divisors.
The line bundle $-K_X=2A+3B-E$ has degree 1 on all three curves
$C,D_1,D_2$. So every ample line bundle is $-K_X$
plus a nef divisor, hence $-K_X$ plus an $\N$-linear combination
of the three divisors mentioned.

A general divisor in $|A|$ is $\P^2$, which is toric.
A general divisor in $|B|$
is the blow-up of the quadric surface $\P^1\times \P^1$
at two collinear points, which is toric. A general divisor
in $|A+2B-E|$ is $\P^2$ blown up at 4 points with no 3 collinear,
hence a quintic del Pezzo surface. That is not toric, but it satisfies
Bott vanishing.
By Lemma \ref{induction}, this reduces Bott vanishing for (3.22)
to the single line bundle $-K_X$.

Since $X$ is rigid, we know that $H^j(X,\Omega^2_X\otimes K_X^*)
=H^j(X,TX)$ is zero for $j>0$. It remains to show that
$H^j(X,\Omega^1_X\otimes K_X^*)$ is zero for $j>0$. 
Let $Y=\P^1\times \P^2$;
then $-K_Y=2A+3B$.
The curve $F$ is a complete intersection $S_1\cap S_2$ in $Y$
with $S_1\sim A$ and $S_2=2B$. So $-K_Y$, $-K_Y-S_1$,
and $-K_Y-S_2$ are ample.
By Lemma \ref{ci}, the desired vanishing would follow
if we have $H^j(Y,\Omega^1_Y(-K_Y-S_1-S_2))=0$ for $j>1$.
Here $-K_Y-S_1-S_2=A+2B$ is ample, and so this cohomology
is actually zero for $j>0$ (Theorem \ref{toric}).
That completes
the proof of Bott vanishing for (3.22).

We now prove Bott vanishing for (3.23). The proof works
in any characteristic. Here $X$ is the blow-up
of $Y:=V_7=\Bl_p\P^3$ along the strict transform $F$
of a conic through the point $p\in \P^3$. We can take
$p$ to be $[1,0,0,0]$ and the conic to be $\{x_3=0, x_1^2-x_0x_2=0\}$.
We have $\Pic(Y)=\Z\{H,E_1\}$ with $-K_Y=4H-2E_1$,
and so $\Pic(X)=\Z\{H,E_1,E_2\}$ and $-K_X=4H-2E_1-E_2$.
Let $C_1$ be the strict transform in $E_1\subset X$ of a line
through the point $E_1\cap F$ in $Y$. (I am using the same
name for the surface $E_1\cong \P^2$ in $Y$
and its strict transform in $X$.)
Let $C_2$ be a fiber of $E_2\to F$. Let $D$ be the strict
transform of a line through $p$ in $\P^3$ that meets the conic
at another point.
We have the intersection numbers:

\begin{tabular}{c|rrr}
& $C_1$ & $C_2$ & $D$\\
\hline
$H$    & 0 & 0 & 1 \\
$E_1$  & $-1$ & 0 & 1\\
$E_2$  & 1 & $-1$ & 1\\
\end{tabular}

We compute that the dual basis to $C_1,C_2,D$
is $H-E_1, 2H-E_1-E_2,H$. 
These three divisors are basepoint-free, hence nef,
giving contractions of $X$ to
$\P^2$, the quadric 3-fold $Q$, and $\P^3$.
It follows that $\overline{\Curv(X)}=\R^{\geq 0}\{C_1,C_2,D\}$,
and the nef cone is spanned by the three divisors mentioned.
By the intersection numbers, 
the nef monoid is also generated by those three divisors.
The line bundle $-K_X=4H-2E_1-E_2$ has degree 1 on all three curves
$C_1,C_2,D$. So every ample line bundle is $-K_X$
plus a nef divisor, hence $-K_X$ plus an $\N$-linear combination
of the three divisors mentioned.

A general divisor in $|H|$ or $|H-E|$ is $\P^2$ blown up
at two points, which is toric.
A general divisor in $|2H-E_1-E_2|$
is a quadric surface $\P^1\times \P^1$ blown
up at one point, which is also toric.
By Lemma \ref{induction}, this reduces Bott vanishing for (3.23)
to the single line bundle $-K_X$.

Since $X$ is rigid, we know that $H^j(X,\Omega^2_X\otimes K_X^*)
=H^j(X,TX)$ is zero for $j>0$. It remains to show that
$H^j(X,\Omega^1_X\otimes K_X^*)$ is zero for $j>0$.
We have $Y\cong \Bl_p \P^3$ and $-K_Y=4H-2E_1$.
The curve $F$ is a complete intersection $S_1\cap S_2$ in $Y$
with $S_1\sim H-E_1$ and $S_2\sim 2H-E_1$. So $-K_Y$, $-K_Y-S_1=
3H-E_1$, and $-K_Y-S_2=2H-E_1$ are ample. By Lemma \ref{ci},
the desired vanishing would follow if we have $H^j(Y,\Omega^1_Y
(-K_Y-S_1-S_2))=0$ for $j>1$. Here $-K_Y-S_1-S_1=H$ is nef,
and so that follows from Proposition \ref{nef}.
This completes
the proof of Bott vanishing for (3.23).

\section{The Fano 3-folds (4.3), (4.4), (4.5), (4.6), (4.7), (4.8)}

For these Fano 3-folds with Picard number 4, the proof
of Bott vanishing is again efficient by our methods.
Each variety is the blow-up of a smooth toric variety along a smooth curve.

Let us prove Bott vanishing for the Fano 3-fold (4.3).
We work in characteristic not 2, since $X$ is not rigid
in characteristic 2.
Here $X$ is the blow-up of $(\P^1)^3$ along a curve $F$
of degree $(1,1,2)$. 
We can take
$F=\{x_0y_1=x_1y_0,x_0^2z_1=x_1^2z_0\}$
in $(\P^1)^3=\{ ([x_0,x_1],
[y_0,y_1],[z_0,z_1])\}$.
We have $\Pic(X)=\Z\{A,B,C,E\}$ and $-K_X=2A+2B+2C-E$.
Let $D_4$ be a fiber of $E\to F$. Let $D_1,D_2,D_3$ be strict
transforms of curves $\P^1\times\pt\times\pt$,
$\pt\times\P^1\times\pt$, $\pt\times\pt\times \P^1$
that meet the curve $F$ in one point each.
We have the intersection numbers:

\begin{tabular}{c|rrrr}
& $D_1$ & $D_2$ & $D_3$ & $D_4$ \\
\hline
$A$  & 1 & 0 & 0 & 0 \\
$B$  & 0 & 1 & 0 & 0\\
$C$  & 0 & 0 & 1 & 0\\
$E$  & 1 & 1 & 1 & $-1$\\
\end{tabular}

We compute that the dual basis to $D_1,D_2,D_3,D_4$
is $A,B,C,A+B+C-E$. 
These three divisors are basepoint-free, hence nef,
using the equation for $F$ above. They give
contractions of $X$ to $\P^1$, $\P^1$, $\P^1$, and $\P^2$.
It follows that $\overline{\Curv(X)}=\R^{\geq 0}\{D_1,D_2,D_3,D_4\}$,
and the nef cone is spanned by the four divisors mentioned.
By the intersection numbers, 
the nef monoid is also generated by those four divisors.
The line bundle $-K_X=2A+2B+2C-E$ has degree 1 on all four curves
$D_1,D_2,D_3,D_4$. So every ample line bundle is $-K_X$
plus a nef divisor, hence $-K_X$ plus an $\N$-linear combination
of the four divisors mentioned.

A general divisor in $|A|$ or $|B|$ is $\P^1\times\P^1$
blown up at one point, which is toric.
A general divisor in $|C|$
if $\P^1\times\P^1$ blown up at 2 points, which is toric.
A general divisor
in $|A+B+C-E|$ is $\P^2$ blown up at 3 non-collinear points,
which is toric.
By Lemma \ref{induction}, this reduces Bott vanishing for (4.3)
to the single line bundle $-K_X$.

Since $X$ is rigid, we know that $H^j(X,\Omega^2_X\otimes K_X^*)
=H^j(X,TX)$ is zero for $j>0$. It remains to show that
$H^j(X,\Omega^1_X\otimes K_X^*)$ is zero for $j>0$. Let $Y=(\P^1)^3$,
which has $-K_Y=2A+2B+2C$.
The curve $F$ is a complete intersection $S_1\cap S_2$ in $Y$,
with $S_1\sim A+B$ and $S_2\sim 2A+C$.
By Lemma \ref{ci}, the desired vanishing would follow if
$-K_Y$ and $-K_Y-S_1$ are ample, the surface $S_1$ is toric,
$(-K_Y-S_2)|_{S_1}$ is ample, and $(-K_Y-S_1-S_2)|_{S_1}$ is nef.
Indeed, $-K_Y$ and $-K_Y-S_1=A+B+2C$ are ample, and $S_1=\{x_0y_1=x_1y_0\}
\subset (\P^1)^3$ is isomorphic
to $(\P^1)^2$, which is toric. The line bundles $A$ and $B$ become
isomorphic on $S_1$, and the nef cone of $S_1$ is spanned by $A$ and $C$.
So $(-K_Y-S_1)|_{S_1}=2A+C$ is ample, and
$(-K_Y-S_1-S_2)|_{S_1}=C$ is nef.
That completes
the proof of Bott vanishing for (4.3).

We first prove Bott vanishing for (4.4). 
We work in characteristic not 2, since $X$ is not rigid
in characteristic 2.
Here $X$ is the blow-up
of $\Bl_{p_2,p_3}Q$ (with $p_2$ and $p_3$ non-collinear points
on the quadric $Q$)
along the strict transform of a conic through $p_2$ and $p_3$.
We will use a different description of this variety,
mentioned by Mori-Mukai \cite[Table 3]{MM}. Namely,
let $Y_1$ be the blow-up of $\P^3$ along a line $F_1$, and let
$Y$ be the blow-up of $Y_1$ along the inverse image $F_2$
of a point $p\in F_1$.
Let $F_3$ be the inverse image in $Y$ of a conic in $\P^3$
disjoint from the line $F_1$. Then $X$ is the blow-up of $Y$ along $F_3$.
This description has the advantage that $Y$ is a toric variety.
We can take the point $p$ in $\P^3$ to be $[0,0,1,0]$, the line $F_1$
to be $\{x_0=x_1=0\}$, and the conic $F_3$ to be $\{x_2=x_3,
x_0x_1-x_2^2=0\}$.

Let $E_1\to F_1$ be the exceptional divisor in $Y_1$; we also write
$E_1$ for its strict transform in $Y$ or $X$. Let $E_2\to F_2$
be the exceptional divisor in $Y$, or its strict transform in $X$.
Let $E_3\to F_3$ be the exceptional divisor in $X$.
Then $\Pic(X)=\Z\{H,E_1,E_2,E_3\}$ and
$-K_X=4H-E_1-2E_2-E_3$. (To check the formula
for $-K_X$, note that the pullback of $E_1\subset Y_1$ is
$E_1+E_2$ in $X$.)
For $i=1,2,3$, let
$C_i$ be a general fiber of $E_i\to F_i$.
Let $D_1$ be the strict transform of a line in $\P^3$ through $p$
and a point of the conic $F_3$. Let $q$ in $\P^3$ be the intersection
of the line $F_1$ with the plane containing the conic $F_3$.
Let $D_2$ be the strict transform
of a line in $\P^3$ through $q$ that meets the conic $F_3$ in 2 points.
We have the intersection numbers:

\begin{tabular}{c|rrrrr}
& $C_1$ & $C_2$ & $C_3$ & $D_1$ & $D_2$\\
\hline
$H$    & 0  & 0  & 0  & 1 & 1 \\
$E_1$  &$-1$& 1  & 0  & 0 & 1\\
$E_2$  & 0  &$-1$& 0  & 1 & 0\\
$E_3$  & 0  & 0  &$-1$& 1 & 2\\
\end{tabular}

Using Magma, we compute that the dual cone to $\R^{\geq 0}
\{C_1,C_2,C_3,D_1,D_2\}$ is spanned by $H,H-E_2,H-E_1-E_2,
2H-E_3,2H-E_2-E_3$.
These five divisors are basepoint-free, hence nef,
giving contractions of $X$ to $\P^3$, $\P^2$, $\P^1$, the quadric
3-fold $Q$, and $\P^3$.
It follows that $\overline{\Curv(X)}=\R^{\geq 0}\{C_1,C_2,C_3,D_2,D_3\}$,
and the nef cone is spanned by the five divisors mentioned.
More strongly, Magma checks that the nef monoid
is generated by these five divisors.
The line bundle
$-K_X=4H-E_1-2E_2-E_3$
has degree 1 on all five curves
$C_1,C_2,C_3,D_1,D_2$. So every ample line bundle is $-K_X$
plus a nef divisor, hence $-K_X$ plus an $\N$-linear combination
of the five divisors mentioned.

For all five divisors, a general divisor in the linear system
is $\P^2$
blown up at 3 non-collinear points, which is toric.
By Lemma \ref{induction}, this reduces Bott vanishing for (4.4)
to the single line bundle $-K_X$.

Since $X$ is rigid, we know that $H^j(X,\Omega^2_X\otimes K_X^*)
=H^j(X,TX)$ is zero for $j>0$. It remains to show that
$H^j(X,\Omega^1_X\otimes K_X^*)$ is zero for $j>0$.
The curve $F_3$ in $Y$ is a complete intersection
$F_1=S_1\cap S_2$, where $S_1\sim H$ and $S_2\sim 2H$.
(Here, in $\P^3$, $S_1$ is the plane $\{x_2=x_3\}$ and $S_1$
is the quadric cone $\{x_0x_1-x_3^2=0\}$, singular at $p$.
Since the point $p$ is disjoint from the plane $S_1$ in $\P^3$, we could
also describe $F_3$ in $Y$ as a complete intersection of $H$ and $2H-E_1$,
or of $H$ and $2H-2E_2$. But we choose the description mentioned.)

By Lemma \ref{ci}, the desired vanishing would follow
if $-K_Y$, $-K_Y-S_1$ are ample, $S_1$ is a toric variety,
$(-K_Y-S_2)|_{S_1}$ is ample,
and $(-K_Y-S_1-S_2)|_{S_1}$ is nef. Here $-K_Y=4H-E_1-2E_2$,
and the nef cone of $Y$ is spanned by $H,H-E_2,H-E_1-E_2$.
So $-K_Y$ and $-K_Y-S_1$ are ample. The surface $S_1$
is $\P^2$ blown up at one point,
which is toric. Since $E_2$ is disjoint from $S_1$,
$-K_Y-S_2=2H-E_1-2E_2$ restricted to $S_1$ is numerically equivalent
to $(1/2)H+(1/2)(H-E_2)+(H-E_1-E_2)$, which is ample;
and $-K_Y-S_1-S_2=H-E_1-2E_2$ restricted to $S_1$
is numerically equivalent to $H-E_1-E_2$, which is nef.
That completes the proof
of Bott vanishing for (4.4).

Next, we prove Bott vanishing for (4.5). 
We work in characteristic not 2, since $X$ is not rigid
in characteristic 2.
Here $X$ is the blow-up
of $Y:=\P^1\times \P^2$ along two disjoint curves,
$F_1$ of degree $(2,1)$ and $F_2$ of degree $(1,0)$.
Thus $F_1$ is contained in $\P^1\times l$, for a line $l$ in $\P^2$,
and $F_2$ is equal to $\P^1\times p$, for a point $p\not\in l$.
We can take $F_1=\{y_0=0,x_1y_1^2=x_0y_2^2\}$
and $F_2=\{y_1=0,y_2=0\}$ in $\P^1\times \P^2=\{ ([x_0,x_1],
[y_0,y_1,y_2])\}$.
We have $\Pic(X)=\Z\{A,B,E_1,E_2\}$ and $-K_X=2A+3B-E_1-E_2$.
Let $C_i\subset E_i$ be a fiber of $E_i\to F_i$, for $i=1,2$.
Let $D_1$ be the strict transform of the curve $\P^1\times q$ in $Y$,
for a point $q\in l$. Let $(s,t)$ be a point in $F_2$, and let $l_2$
be the line through $p$ and $t$ in $\P^2$. Then let $D_2$
be the strict transform of the curve $s\times l_2\subset \P^1\times \P^2$.
Finally, let $D_3$ be the strict transform of the curve $s\times l$.
We have the intersection numbers:

\begin{tabular}{c|rrrrr}
& $C_1$ & $C_2$ & $D_1$ & $D_2$ & $D_3$\\
\hline
$A$  & 0 & 0 & 1 & 0 & 0\\
$B$  & 0 & 0 & 0 & 1 & 1\\
$E_1$& $-1$& 0 & 1 & 1 & 2\\
$E_2$& 0 & $-1$& 0 & 1 & 0\\
\end{tabular}

We compute that the dual cone to $\R^{\geq 0}\{C_1,C_2,D_1,D_2,D_3\}$
is spanned by the five divisors $A$, $B$, $B-E_2$, $A+2B-E_1$,
and $A+2B-E_1-E_2$. These are all basepoint-free, hence nef,
using the equations for $F_1$ and $F_2$ above. They give
contractions of $X$ to $\P^1$, $\P^2$, $\P^1$,
a nodal 3-fold in $\P^6$, and a nodal quadric 3-fold in $\P^4$.
It follows that $\overline{\Curv(X)}=\R^{\geq 0}\{C_1,C_2,D_1,D_2,D_3\}$,
and the nef cone is spanned by the five divisors mentioned.
Using Magma, we also compute
that the nef monoid is generated by those five divisors.
The line bundle $-K_X=2A+3B-E_1-E_2$ has degree 1 on all five curves
$C_1,C_2,D_1,D_2,D_3$. So every ample line bundle is $-K_X$
plus a nef divisor, hence $-K_X$ plus an $\N$-linear combination
of the five divisors mentioned.

A general divisor in $|A|$ is $\P^2$ blown up at 3 non-collinear
points, which is toric. A general divisor in $|B|$ or in $|B-E_2|$
is $\P^1\times \P^1$
blown up at one point, which is toric. A general divisor
in $|A+2B-E_1-E_2|$ is $\P^2$ blown up at 4 points with no 3 collinear,
hence a quintic del Pezzo surface. That is not toric, but it satisfies
Bott vanishing. A general divisor in $|A+2B-E_1|$ is the previous
surface blown up at one more point; that does not satisfy Bott vanishing.
Instead, we can observe that $A+2B-E_1=(A+2B-E_1-E_2)+E_2$,
where $E_2$ is a $\P^1$-bundle over $\P^1$, which is toric.
By Lemma \ref{induction}, that reduces Bott vanishing for (4.5)
to the single line bundle $-K_X$.

Since $X$ is rigid, we know that $H^j(X,\Omega^2_X\otimes K_X^*)
=H^j(X,TX)$ is zero for $j>0$. It remains to show that
$H^j(X,\Omega^1_X\otimes K_X^*)$ is zero for $j>0$.
Let $Z$ be the blow-up of $Y=\P^1\times \P^2$ along the curve
$F_2$ of degree $(0,1)$; then $Z$ is a toric variety,
with $\Pic(Z)=\Z\{A,B,E_2\}$ and $\Nef(X)=\R^{\geq 0}
\{A,B,B-E_2\}$. The curve $F_1$ is a complete intersection
$S_1\cap S_2$ in $Z$, with $S_1\sim B$ and $S_2\sim A+2B$.
By Lemma \ref{ci}, the desired vanishing holds
if $-K_Z$ and $-K_Z-S_1$ are ample, $S_1$ is a toric variety,
$(-K_Y-S_2)_{S_1}$ is ample, and $(-K_Y-S_1-S_2)_{S_1}$ is nef.
Indeed, $-K_Z=2A+3B-E_2$ and $-K_Z-S_1=2A+2B-E_2$ are ample.
The surface $S_1=\{y_0=0\|$ is isomorphic to $\P^1\times \P^2$,
and the exceptional divisor $E_2$ in $Z$ is disjoint from $S_1$.
So $(-K_Y-S_2)|_{S_1}=A+B$ is ample on $S_1$,
and $(K_Y-S_1-S_2)|_{S_1}=A$ is nef on $S_1$.
That completes
the proof of Bott vanishing for (4.5).

Next, we prove Bott vanishing for (4.6). The proof works
in any characteristic. Here $X$ is the blow-up
of $\P^3$ along three disjoint lines, $F_1,F_2,F_3$. We have
$\Pic(X)=\Z\{H,E_1,E_2,E_3\}$ and $-K_X=4H-E_1-E_2-E_3$.
Let $C_i$ be a fiber of $E_i\to F_i$ for $i=1,2,3$,
and let $D$ be the strict transform of a line in $\P^3$
meeting $F_1$, $F_2$, and $F_3$ (which exists).
Then $H\cdot C_i=0$, $H\cdot D=1$, $E_i\cdot C_j=-\delta_{ij}$,
and $E_i\cdot D=1$, for $1\leq i,j\leq 3$. 
We have the intersection numbers:

\begin{tabular}{c|rrrr}
& $C_1$ & $C_2$ & $C_3$ & $D$ \\
\hline
$H$& 0 & 0 & 0 & 1\\
$E_1$&$-1$& 0 & 0 & 1\\
$E_2$& 0 & $-1$& 0 & 1\\
$E_3$& 0 & 0 & $-1$& 1\\
\end{tabular}

It follows
that the dual basis to $C_1,C_2,C_3,D\in N_1(X)$
is $H-E_1,H-E_2,H-E_3,H$. These four divisors
are basepoint-free, hence nef, giving contractions
of $X$ to $\P^1$, $\P^1$, $\P^1$, and $\P^3$.
Therefore, $\overline{\Curv(X)}=\R^{\geq 0}\{C_1,C_2,C_3,D\}$
and the nef cone is spanned by $H-E_1,H-E_2,H-E_3,H$.
By the intersection numbers, the nef monoid in $\Pic(X)$ is spanned by those
four divisors. Also, the line bundle $-K_X=4H-E_1-E_2-E_3$ has degree
1 on the four curves $C_1,C_2,C_3,D$. It follows that every ample
line bundle on $X$ is $-K_X$ plus a nef divisor, hence $-K_X$
plus an $\N$-linear combination of $H-E_1,H-E_2,H-E_3,H$.

For $1\leq i\leq 3$, a general divisor in $|H-E_i|$ is the blow-up
of $\P^2$ at 2 points, which is toric. A general divisor in $|H|$
is the blow-up of $\P^2$ at 3 non-collinear points, which is also
toric. Therefore, Lemma \ref{induction} reduces Bott vanishing
for $X$ to the single line bundle $-K_X$.
Since $X$ is rigid, we know that $H^j(X,\Omega^2_X\otimes K_X^*)
=H^j(X,TX)$ is zero for $j>0$. It remains to show that
$H^j(X,\Omega^1_X\otimes K_X^*)=0$ for $j>0$.
Let $Y$ be the blow-up of $\P^3$ along the lines $F_1$ and $F_2$;
then $Y$ is a toric variety. We have $\Pic(Y)=\Z\{H,E_1,E_2\}$
and $\Nef(Y)=\R^{\geq 0}\{H-E_1,H-E_2,H\}$. The curve $F_3$
in $Y$ is a complete intersection $S_1\cap S_2$ with
$S_1\sim H$ and $S_2\sim H$. So $-K_Y=4H-E_1-E_2$,
$-K_Y-S_1=3H-E_1-E_2$, and $-K_Y-S_2=3H-E_1-E_2$ are ample,
and $-K_Y-S_1-S_2=2H-E_1-E_2$ is nef. By Lemma \ref{ci},
it follows that $H^j(X,\Omega^1_X\otimes K_X^*)=0$ for $j>0$.
That completes
the proof of Bott vanishing for (4.6).

We now prove Bott vanishing for (4.7).
The proof works in any characteristic.
Here $X$ is the blow-up of the flag
manifold $W\subset \P^2\times \P^2$ along disjoint curves
$F_1$ of degree $(0,1)$ and $F_2$ of degree $(1,0)$. We will instead use
a different description, mentioned by Mori-Mukai \cite[Table 3]{MM}.
Let $S$ be the blow-up of $\P^2$ at a point $p$. Embed $F=\P^1$
into $Y:=\P^1\times S$ by the identity map on $\P^1$ and the inclusion
into $S$ as the strict transform of a line in $\P^2$ not containing $p$.
Then $X$ is the blow-up of $Y$ along $F$. (The advantage of this
description, for hand calculation,
is that it expresses $X$ as the blow-up of the toric
variety $Y$ along a single curve.) We can take $S$ to be $\P^2$
blown up at the point $[0,0,1]$, with $F$ defined by
$y_2=0,x_0y_1=x_1y_0$ in $\P^1\times \P^2=\{[x_0,x_1],
[y_0,y_1,y_2]\}$.

Let $A$ be the pullback to $X$ of $O(1)$ on $\P^1$.
Let $B$ and $H$ be the pullbacks of $O(1)$ by the contractions
of $S$ to $\P^1$ and $\P^2$. Then $\Pic(X)=\Z\{A,B,H,E\}$
and $-K_X=2A+B+2H-E$. Let $C$ be the strict transform in $X$
of a curve $\P^1\times \pt$ that meets the curve $F$ in one point.
Let $D$ be the strict transform of a point in $\P^1$ times the $(-1)$-curve
in $S$. (The curve $D$ is disjoint from $G$.) Let $G$ be the strict
transform of a point in $\P^1$ times the strict transform in $S$
of a line in $\P^2$ through $p$ such that $G$ meets $F$ in one point.
Let $K$ be a fiber of the exceptional divisor $E\to F$. 
We have the intersection numbers:

\begin{tabular}{c|rrrr}
& $C$ & $D$ & $G$ & $K$\\
\hline
$A$& 1 & 0 & 0 & 0\\
$B$& 0 & 1 & 0 & 0 \\
$H$& 0 & 0 & 1 & 0\\
$E$& 1 & 0 & 1 &$-1$\\
\end{tabular}

The dual basis to $C,D,G,K$ is given by
$A,B,H,A+H-E$.
These line bundles are basepoint-free,
hence nef, giving contractions to $\P^1$ (twice)
and $\P^2$ (twice). (Using the equation for $G$ above,
a basis for the sections of $A+H-E$ is given by $x_0y_1-x_1y_0,
x_0y_2,x_1y_2$; these equations define the curve $G$ in $Y$ as a scheme,
which proves the basepoint-freeness of $A+H-E$ on $X$.)

It follows that $\overline{\Curv(X)}
=\R^{\geq 0}\{C,D,G,K\}$, and the nef cone is spanned
by $A,B,H,A+H-E$.
More strongly, the nef monoid in $\Pic(X)$ is generated
by these four divisors. Since $-K_X$ has degree 1
on $C$, $D$, $G$, and $K$, every ample line bundle on $X$
is $-K_X$ plus a nef divisor, hence $-K_X$ plus an
$\N$-linear combination of the $A,B,H,A+H-E$.

A general divisor in $|A|$, $|B|$, $|H|$,
or $|A+H-E|$ is isomorphic to
$\P^2$ blown up at 2 points, which is toric.
By Lemma \ref{induction}, this reduces Bott vanishing on $X$
to the ample line bundle $-K_X$.

For $-K_X$ and $\Omega^2_X$, this is easy: since $X$ is rigid,
we have $H^j(X,\Omega^2_X\otimes K_X^*)=H^j(X,TX)=0$
for $j>0$. It remains
to prove Bott vanishing for $-K_X$ and $\Omega^1_X$,
meaning that $H^j(X,\Omega^1_X\otimes K_X^*)=0$ for $j>0$.
Since $Y=\P^1\times S$,
we have $\Nef(Y)=\R^{\geq 0}\{A,B,H\}$.
We can view the curve $F$ is a complete intersection $S_1\cap S_2$
with $S_1\sim H$ and $S_2\sim A+B$. By Lemma \ref{ci},
the desired vanishing holds is $-K_Y$, $-K_Y-S_1$,
and $-K_Y-S_2$ are ample and $-K_Y-S_1-S_2$ are nef.
Indeed, $-K_Y=2A+B+2H$ is ample,
$-K_Y-S_1=2A+B+H$ is ample, $-K_Y-S_2=A+B+H$ is ample,
and $-K_Y-S_1-S_2=A+B$ is nef.
That completes the proof of Bott vanishing for (4.7).

We now prove Bott vanishing for the Fano 3-fold (4.8).
The proof works in any characteristic.
Here $X$ is the blow-up of $(\P^1)^3$ along a curve $F$
of degree $(0,1,1)$. 
We can take
$F=\{x_1=0,y_0z_1=y_1z_0\}$
in $(\P^1)^3=\{ ([x_0,x_1],
[y_0,y_1],[z_0,z_1])\}$.
We have $\Pic(X)=\Z\{A,B,C,E\}$ and $-K_X=2A+2B+2C-E$.
Let $D_4$ be a fiber of $E\to F$. Let $D_1,D_2,D_3$ be strict
transforms of curves $\P^1\times\pt\times\pt$,
$\pt\times\P^1\times\pt$, $\pt\times\pt\times \P^1$
that meet the curve $F$ in one point each.
We have the intersection numbers:

\begin{tabular}{c|rrrr}
& $D_1$ & $D_2$ & $D_3$ & $D_4$ \\
\hline
$A$  & 1 & 0 & 0 & 0 \\
$B$  & 0 & 1 & 0 & 0\\
$C$  & 0 & 0 & 1 & 0\\
$E$  & 1 & 1 & 1 & $-1$\\
\end{tabular}

We compute that the dual basis to $D_1,D_2,D_3,D_4$
is $A,B,C,A+B+C-E$. 
These three divisors are basepoint-free, hence nef,
using the equation for $F$ above. They give
contractions of $X$ to $\P^1$, $\P^1$, $\P^1$, and a nodal
quadric 3-fold.
It follows that $\overline{\Curv(X)}=\R^{\geq 0}\{D_1,D_2,D_3,D_4\}$,
and the nef cone is spanned by the four divisors mentioned.
More strongly,
the nef monoid in $\Pic(X)$ is generated by those four divisors.
The line bundle $-K_X=2A+2B+2C-E$ has degree 1 on all four curves
$D_1,D_2,D_3,D_4$. So every ample line bundle is $-K_X$
plus a nef divisor, hence $-K_X$ plus an $\N$-linear combination
of the four divisors mentioned.

A general divisor in $|A|$ is isomorphic
to $\P^1\times\P^1$, which is toric. A general divisor
in $|B|$ or $|C|$ is $\P^1\times\P^1$
blown up at one point, which is toric.
A general divisor
in $|A+B+C-E|$ is $\P^2$ blown up at 3 non-collinear points,
which is toric.
By Lemma \ref{induction}, this reduces Bott vanishing for (4.8)
to the single line bundle $-K_X$.

Since $X$ is rigid, we know that $H^j(X,\Omega^2_X\otimes K_X^*)
=H^j(X,TX)$ is zero for $j>0$. It remains to show that
$H^j(X,\Omega^1_X\otimes K_X^*)$ is zero for $j>0$. Let $Y=(\P^1)^3$;
then $-K_Y=2A+2B+2C$ and
$\Nef(Y)=\R^{\geq 0}\{A,B,C\}$. The curve $F$ is a complete
intersection $S_1\cap S_2$ in $Y$ with $S_1\sim A$ and $S_2\sim B+C$.
By Lemma \ref{ci},
the desired vanishing holds if $-K_Y$, $-K_Y-S_1$, and $-K_Y-S_2$
are ample, and $-K_Y-S_1-S_2$ is nef. In this case, all four line bundles
are ample. That completes
the proof of Bott vanishing for (4.8).

\section{The Fano 3-fold (5.1)}
\label{iterated}

The Fano 3-fold (5.1) (with Picard number 5)
turns out to be the hardest, for proving
Bott vanishing.
Our technique of reducing to general properties
of toric varieties seems not to be strong enough
in this case. Still, we give a meaningful proof
using Hodge cohomology.
We will prove Bott vanishing in characteristic not 2, since $X$ is not rigid
in characteristic 2.

One construction of (5.1) is similar to (4.4).
Let $J$ be a conic in $Q$. Then $X$
is the blow-up of $\Bl_{J} Q$
along three fibers of the exceptional divisor.
As suggested by Coates-Corti-Galkin-Kasprzyk, we instead
view $X$ as a hypersurface in a smooth toric 4-fold $G$
\cite[section 98]{CCGK}. Namely, $G$ is obtained
from $\P^4$ by blowing up along a plane $\Pi$
and then along the fibers over
three non-collinear points $p_2,p_3,p_4$ in $\Pi$. Let $H$ be the pullback
of $O(1)$ on $\P^4$, $E_1$ the (irreducible) exceptional divisor
over $\Pi$, and $E_2,E_3,E_4$ the exceptional divisors over the three points
in $\Pi$. Then $\Pic(G)=\Z\{H,E_1,E_2,E_3,E_4\}$, and the nef cone of $G$
is spanned by $H$, $H-E_2$, $H-E_3$, $H-E_4$, $H-E_1-E_2-E_3-E_4$,
and $2H-E_2-E_3-E_4$, by \cite{CCGK}. (In their notation,
$A=H-E_1-E_2-E_3-E_4$, $B=E_1$, $C=E_2$, $D=E_3$, and $E=E_4$.)
These divisors are basepoint-free, giving contractions of $G$
to $\P^4$, $\P^3$ (three times), $\P^1$, and another toric 4-fold.

For completeness, let us list the intersection numbers
between divisors and some curves on $G$ (which span the cone
of curves). This could be used
to compute the nef cone of $G$, if we did not already know it.
Namely, let $C_1,C_2,C_3,C_4$ be general fibers of the exceptional divisors
$E_1,E_2,E_3,E_4$ on $G$. Also, for $2\leq i\leq 4$, $D_i$
will be a curve in $X$ mapping to the line through $p_j$ and $p_k$,
where $\{i,j,k\}=\{2,3,4\}$; more precisely, let $D_i$ be the section
of $E_1\to \Pi$ associated to a general plane in $\P^2$ containing
that line. Then we have the intersection numbers:

\begin{tabular}{c|rrrrrrrr}
& $C_1$ & $C_2$ & $C_3$ & $C_4$ & $D_2$ & $D_3$ & $D_4$\\
\hline
$H$    & 0 &  0 & 0  & 0 & 1 & 1   & 1 \\
$E_1$  &$-1$& 1 & 1 & 1 &$-1$&$-1$&$-1$\\
$E_2$  & 0 &$-1$ & 0 & 0 & 0 & 1 & 1\\
$E_3$  & 0 & 0 &$-1$& 0 & 1 & 0 & 1\\
$E_4$  & 0 & 0 & 0 &$-1$& 1 & 1 & 0
\end{tabular}

The Fano 3-fold $X$ is a general divisor in the linear system
$|2H-E_2-E_3-E_4|$; that is, $X$ is the strict transform in $G$
of a quadric 3-fold $Q$ containing $p_2,p_3,p_4$. So $X$ is obtained
from $Q$ by blowing up
along a conic $F_1$ containing those 3 points, and then along
the fibers over those 3 points.
Here $\Pic(X)=\Z\{H,E_1,E_2,E_3,E_4\}$,
but (unfortunately)
the nef cone of $X$ turns out to be bigger than $\Nef(G)$.
To list some curves on $X$: we can view $C_1,C_2,C_3,C_4$ above
as curves on $X$, namely fibers in the four exceptional divisors.
For $i=2,3,4$, let $K_i$ be the strict transform in $X$ of a line
in $Q$ through $p_i$. Finally, let $V$ be the section of $E_1\to F_1$
associated to a hyperplane section of $Q$ containing the conic $F_1$.
Then we have the intersection numbers:

\begin{tabular}{c|rrrrrrrr}
& $C_1$ & $C_2$ & $C_3$ & $C_4$ & $K_2$ & $K_3$ & $K_4$ & $V$\\
\hline
$H$    & 0 &  0 & 0  & 0 & 1 & 1   & 1 & 2\\
$E_1$  &$-1$& 1 & 1 & 1 & 0 & 0 & 0 & $-1$\\
$E_2$  & 0 &$-1$ & 0 & 0 & 1 & 0 & 0 & 1\\
$E_3$  & 0 & 0 &$-1$& 0 & 0 & 1 & 0 & 1\\
$E_4$  & 0 & 0 & 0 &$-1$& 0 & 0 & 1 & 1
\end{tabular}

Using Magma, we compute that the dual cone to the cone spanned
by these eight curves is spanned by $H$, $H-E_2$, $H-E_3$,
$H-E_4$, $H-E_2-E_3$, $H-E_2-E_4$, $H-E_3-E_4$,
and $H-E_1-E_2-E_3-E_4$.
These eight divisors are basepoint-free, hence nef,
giving contractions of $X$ to the quadric 3-fold $Q$,
$\P^3$ (three times), $\P^2$ (three times), and $\P^1$.
It follows that the cone of curves of $X$ is spanned by the eight curves
above,
and the nef cone is spanned by these eight divisors.
More strongly, Magma checks that the nef monoid in $\Pic(X)$
is generated by these eight divisors.
The line bundle $-K_X=3H-E_1-2E_2-2E_3-2E_4$
has degree 1 on all eight curves.
So every ample line bundle is $-K_X$
plus a nef divisor, hence $-K_X$ plus an $\N$-linear combination
of the eight divisors mentioned.

For each of those divisors except $H-E_1-E_2-E_3-E_4$,
a general divisor in the linear system
is isomorphic to $\P^2$
blown up at three non-collinear points, which is toric.
A general divisor in $|H-E_1-E_2-E_3-E_4|$ is $\P^2$
blown up at four points with no three collinear; this is the quintic
del Pezzo surface, which is not toric but satisfies Bott vanishing.
By Lemma \ref{induction}, this reduces Bott vanishing for (5.1)
to the single line bundle $-K_X$.

Since $X$ is rigid in characteristic not 2,
we know that $H^j(X,\Omega^2_X(-K_X))
=H^j(X,TX)$ is zero for $j>0$.
It remains to show that $H^j(X,\Omega^1_X(-K_X))=0$ for $j>0$,
which we will prove in any characteristic.
This was shown by Belmans-Fatighenti-Tanturri, in terms
of the isomorphic vector bundle $\Lambda^2TX$, when the base field
has characteristic zero \cite[Appendix A]{BFT}.

To do this, recall that $X$ is a hypersurface in the smooth toric
4-fold $G$.
We have $-K_G=2(2H-E_2-E_3-E_4)+(H-E_1-E_2-E_3-E_4)
=5H-E_1-3E_2-3E_3-3E_4$,
which is nef and big but not ample. So $-K_X$ is the restriction
of $-K_G-X=3H-E_1-2E_2-2E_3-2E_4$, which is also nef and big but not ample.
Since $G$ is a toric variety, it follows that $-K_G-X$ is basepoint-free
\cite[p.~68]{Fulton}.
Consider the contraction $\pi\colon G\to F$ associated
to the line bundle $-K_G-X$; this is a small contraction, and all
contracted curves are disjoint from $X$.
The singular set $S$ of $F$
consists of three disjoint $\P^1$'s. (The curves $D_2,D_3,D_4$
in $G$ are contracted by $\pi$ to points in these three components of $S$.)
Here $X$ is still
a smooth hypersurface in $F$, but now $-K_F$ and $X=2H-E_2-E_3-E_4$
(which pull back to $-K_G$ and $X$ on $G$)
are ample line bundles on $F$. The description
of the toric Fano 4-fold $F$ by a fan
is given in Belmans-Fatighenti-Tanturri's
file about the Fano 3-fold (5.1) \cite{BFT}.
The singularities of $F$
are locally isomorphic to a smooth curve times the 3-fold node.
We will prove the desired cohomology
vanishing by relating $X$ to the singular toric Fano 4-fold $F$,
although the smooth toric 4-fold $G$ also comes up
in the argument.

We have the exact sequences of coherent sheaves
$0\to O_X(-X)\to \ohat|_X \to \Omega^1_X\to 0$
on $X$ and $0\to O(-X)\to O_F \to O_X\to 0$ on $F$.
Tensoring the first sequence with $-K_F-X$, we have
$0\to O_X(-K_F-2X)\to \ohat(-K_F-X)|_X\to \Omega^1_X(-K_X)\to 0$.
So the desired vanishing would follow
if $\ohat(-K_F-X)|_X$ has zero cohomology in positive degrees
and $O_X(-K_F-2X)$ has zero cohomology in degrees $>1$.
From the second sequence above, it suffices to show that
(1) $\ohat(-K_F-X)$ has zero cohomology (on $F$) in positive degrees;
(2) $\ohat(-K_F-2X)$ has zero cohomology in degrees $>1$;
(3) $-K_F-2X$ has zero cohomology (on $F$) in degrees $>1$; and
(4) $-K_F-3X$ has zero cohomology in degrees $>2$. Since $-2K_F-2X$
and $-2K_F-3X$
is ample, (3) and (4) are immediate from Kodaira vanishing
on the toric variety $F$ (part of Theorem \ref{log}). Also, since
$-K_F-X$ is ample, (1) follows from Bott vanishing on $F$
(Theorem \ref{log}). (This is the advantage of working with $F$
rather than $G$.)

Here $A:=-K_F-2X=H-E_1-E_2-E_3-E_4$
is nef, but $F$ is singular, and so (2) is not immediate
from Proposition \ref{nef}. On the other hand, we know that
$H^j(G,\Omega^1_G(A))=0$ for $j>0$ by Proposition \ref{nef},
and so it seems natural to compare the singular variety $F$ with
its resolution $G$. Let $S$ be the singular locus of $F$,
which is the disjoint union of three $\P^1$'s.
Near $S$, the morphism $\pi\colon G\to F$ is locally
a smooth curve times
one of the two small resolutions
of the 3-fold node $xy=zw$ in $A^4$. It follows, for example
using the theorem on formal functions \cite[Theorem III.11.1]{Hartshorne},
that the sheaf
$R^j\pi_*\Omega^1_G$ is isomorphic to $\ohat$ for $j=0$, $O_S$ for $j=1$,
and zero otherwise. Equivalently, we have an exact triangle
$\ohat \to R\pi_*\Omega^1_G\to O_S[-1]$ in the derived
category of $F$. So we have a long exact sequence
$$H^1(F,\ohat)\to H^1(G,\Omega^1_G)\to H^0(S,O_S)\to H^2(F,\ohat)\to\cdots .$$

Here $H^0(S,O_S)\cong k^3$. I claim that the map from $H^1(G,\Omega^1_G)$
to $H^0(S,O_S)$ is surjective. It is equivalent to show that the image
of $H^1(F,\ohat)\to H^1(G,\Omega^1_G)$ has codimension at least 3.
So it suffices to find a surjection
$H^1(G,\Omega^1_G)\to k^3$ that is zero on the image of $H^1(F,\ohat)$.
In the notation above, the curves $D_2,D_3,D_4$ in $G$
map to $k$-points in the three
components of $S$. Then the restriction map from $H^1(G,\Omega^1_G)$
to $\oplus_{i=2}^4H^1(D_i,\Omega^1_{D_i})=k^3$ vanishes on $H^1(F,\ohat)$
(because the curves $D_i$ map to points in $F$). So it suffices to show
that the composition $\Pic(G)\to H^1(G,\Omega^1_G)\to k^3$
is surjective. This map gives the degrees of line bundles on $G$
on the three curves $D_i$. The intersection numbers
of $H,E_3,E_4$ with these three curves are $\begin{pmatrix}
1 & 1 & 1\\1 & 0 & 1\\1 & 1 & 0\end{pmatrix}$,
which has determinant 1. So, for $k$ of any characteristic,
we have shown that $H^1(G,\Omega^1_G)\to H^0(S,O_S)$ is surjective.
(By the exact sequence above, it follows that $H^2(F,\ohat)=0$.)

We want to compute the related groups $H^j(F,\ohat(A))$.
By the exact triangle 
$\ohat(A) \to R\pi_*\Omega^1_G(A)\to A|_S[-1]$ in the derived
category of $F$, we have a long exact sequence
$$H^1(F,\ohat(A))\to H^1(G,\Omega^1_G(A))\to H^0(S,A)
\to H^2(F,\ohat(A))\to \cdots .$$
The line bundle $A$ has degree 1 on each $\P^1$ component
of $S=S_2\coprod S_3\coprod S_4$, using that each of these curves on $F$
is the image of a curve numerically equivalent to $C_1$ on $G$.
Since $A$ is nef on the toric
variety $G$, it is basepoint-free. It follows that
the restriction $H^0(G,A)\to H^0(S_i,A)\cong k^2$ is surjective
for $i=2,3,4$. Combining this with the previous paragraph,
the composition
$H^1(G,\Omega^1_G)\otimes_k H^0(G,A)\to H^1(G,\Omega^1_G(A))
\to H^0(S,A)$ is surjective. Also, $H^j(G,A)=0$
for $j>0$ and (by Proposition \ref{nef})
$H^j(G,\Omega^1_G(A))=0$ for $j>1$.
Therefore, the exact sequence above shows that
$H^j(F,\ohat(A))=0$ for all $j>1$. This is statement (2), above.
That completes the proof that $H^j(X,\Omega^1_X(-K_X))=0$
for $j>0$.
Thus Bott vanishing holds for the Fano 3-fold (5.1).


\small \sc UCLA Mathematics Department, Box 951555,
Los Angeles, CA 90095-1555

totaro@math.ucla.edu
\end{document}